\documentclass[12pt]{amsart}
\usepackage{amscd,verbatim}
\usepackage{amssymb}
\usepackage[all]{xy}
\usepackage[
colorlinks,linkcolor=blue,citecolor=blue,urlcolor=red]{hyperref}
\usepackage{bm}

\newcommand{\A}{\mathbf{A}}

\renewcommand{\P}{\mathbf{P}}

\newcommand{\sC}{\mathcal{C}}
\newcommand{\sD}{\mathcal{D}}

\newcommand{\sV}{\mathcal{V}}

\newcommand{\ul}[1]{{\underline{#1}}}

\newcommand{\Sm}{\operatorname{\mathbf{Sm}}}
\newcommand{\Sch}{\operatorname{\mathbf{Sch}}}

\newcommand{\by}{\xrightarrow}

\newcommand{\iso}{\by{\sim}}

\newcommand{\pro}[1]{\text{\rm pro}_{#1}\text{\rm--}}

\newcommand{\fin}{{\operatorname{fin}}}
\renewcommand{\o}{{\operatorname{o}}}

\newcommand{\red}{{\operatorname{red}}}

\newcommand{\et}{{\operatorname{\acute{e}t}}}
\newcommand{\tto}{\dashrightarrow}
\newcommand{\inj}{\hookrightarrow}

\newcommand{\id}{{\operatorname{Id}}}

\renewcommand{\lim}{\operatornamewithlimits{\varprojlim}}

\newcommand{\ol}{\overline}

\renewcommand{\phi}{\varphi}
\renewcommand{\epsilon}{\varepsilon}

\newcommand{\MNST}{\operatorname{\mathbf{MNST}}}

\newcommand{\MSm}{\operatorname{\mathbf{MSm}}}

\newcommand{\Bl}{{\mathbf{Bl}}}

\newcommand{\Sq}{{\operatorname{\mathbf{Sq}}}}

\newcommand{\OD}{\mathrm{OD}}
\newcommand{\ctimes}{\times^\mathrm{c}}

\newcommand{\ulMSm}{\operatorname{\mathbf{\underline{M}Sm}}}

\newcommand{\ulMNST}{\operatorname{\mathbf{\underline{M}NST}}}

\newcommand{\ulomega}{\underline{\omega}}

\newcommand{\Comp}{\operatorname{\mathbf{Comp}}}

\newcommand{\MV}{\operatorname{MV}}
\newcommand{\ulMV}{\operatorname{\underline{MV}}}
\newcommand{\ulMVfin}{\operatorname{\underline{MV}^{\mathrm{fin}}}}

\newcounter{spec}
\newenvironment{thlist}{\begin{list}{\rm{(\roman{spec})}}%
{\usecounter{spec}\labelwidth=20pt\itemindent=0pt\labelsep=10pt}}%
{\end{list}}%

\newtheorem{Th}{Theorem}
\newtheorem{lemma}{Lemma}[subsection]
\newtheorem{thm}[lemma]{Theorem}

\newtheorem{prop}[lemma]{Proposition}

\newtheorem{cor}[lemma]{Corollary}

\theoremstyle{definition}
\newtheorem{defn}[lemma]{Definition}

\newtheorem{definition}[lemma]{Definition}

\newtheorem{condition}[lemma]{Conditions}
\theoremstyle{remark}

\newtheorem{Rk}{Remark}

\newtheorem{remark}[lemma]{Remark}
\newtheorem{remarks}[lemma]{Remarks}
\newtheorem{ex}[lemma]{Example}

\newtheorem{claim}[lemma]{Claim}

\numberwithin{equation}{subsection}

\def\Comp{\Comp^{\fin}}

\def\MSm{\operatorname{\mathbf{MSm}}}

\def\ulMSm{\operatorname{\mathbf{\ul{M}Sm}}}

\def\ulMNST{\operatorname{\mathbf{\ul{M}NST}}}

\def\MNST{\operatorname{\mathbf{MNST}}}

\def\Comp{\operatorname{\mathbf{Comp}}}

\begin{document}

\title[Topologies on schemes and modulus pairs]
{Topologies on schemes and modulus pairs}
\author[B. Kahn]{Bruno Kahn}
\address{IMJ-PRG\\Case 247\\
4 place Jussieu\\
75252 Paris Cedex 05\\
France}
\email{bruno.kahn@imj-prg.fr}
\author[H. Miyazaki]{Hiroyasu Miyazaki}
\address{RIKEN iTHEMS, Wako, Saitama 351-0198, Japan}
\email{hiroyasu.miyazaki@riken.jp}
\date{May 13, 2020} 
\thanks{The first author acknowledges the support of Agence Nationale de la Recherche (ANR) under reference ANR-12-BL01-0005. 
The work of the second author is supported by Fondation Sciences Math\'ematiques de Paris (FSMP), by RIKEN Special Postdoctoral Researchers (SPDR) Program, by RIKEN Interdisciplinary Theoretical and Mathematical Sciences Program (iTHEMS), and by JSPS KAKENHI Grant (19K23413).
}

\begin{abstract}
We study relationships between the Nisnevich top\-ol\-ogy on smooth schemes and  certain Grothendieck topologies on proper and not necessarily proper modulus pairs which were introduced respectively in \cite{nistopmod} and \cite{modsheaf1}. Our results play an important role in the theory of sheaves with transfers on proper modulus pairs. 
\end{abstract}


\maketitle

\tableofcontents

\section*{Introduction}

\setcounter{subsection}{1}

In \cite{modsheaf1}, a theory of sheaves on \emph{non-proper} modulus pairs has been studied as a first step to establish the theory of motives with modulus, which is to be a non-$\A^1$-invariant version of Voevodsky's category of motives given in \cite{voetri}. 
This repaired the first part of the mistake in \cite{motmod} (the ancestor of the theory) mentioned in the introduction of \cite{modsheaf1}.

In \cite{modsheaf2}, a theory of sheaves on \emph{proper} modulus pairs is developed as a second step,  thus repairing the second part of the mistake. The main point of these repairs is to prove that the categories $\ulMNST$ and $\MNST$ of  \cite{motmod}, which had been defined in an ad hoc way, are really categories of sheaves (with transfers) for suitable Grothendieck topologies having good formal properties.

The aim of the present paper is to provide some foundational results which will be the key building blocks of the theory in  \cite{modsheaf2}. To explain our aim in more detail, we first recall basic notions of modulus pairs from \cite{modsheaf1}.
We fix a base field $k$ and write $\Sch$ (resp. $\Sm$) for the category of separated $k$-schemes of finite type  (resp. its full subcategory of smooth $k$-schemes).

A \emph{modulus pair} is a pair \[M=(\ol{M},M^\infty),\] where $\ol{M} \in \Sch$ and $M^\infty$ is an effective Cartier divisor on $\ol{M}$ such that the complement of the divisor 
\[M^\o := \ol{M} - M^\infty \]
belongs to $\Sm$.
These conditions imply that $\ol{M}$ is reduced and $M^\o$ is dense \cite[Rem. 1.1.2 (3)]{modsheaf1}. 
We call $\ol{M}$ (resp. $M^\o$) \emph{the ambient space of $M$} (resp. \emph{the interior of $M$}). 

A morphism $f : M \to N$ of modulus pairs is a morphism $f^\o : M^\o \to N^\o$ in $\Sm$ which satisfies the following \emph{admissibility condition}: let $\Gamma$ be the graph of the rational map $\ol{M} \dashrightarrow \ol{N}$ defined by $f^\o$, and let $\Gamma^N \to \Gamma$ be the normalization, whence a diagram $\ol{M} \xleftarrow{a} \Gamma^N \xrightarrow{b} \ol{N}$.
Then $a$ is proper and we have $a^\ast M^\infty \geq b^\ast N^\infty$, where $a^\ast M^\infty $ and $b^\ast N^\infty$ denote the pull-backs of effective Cartier divisors (see \cite[Def. 1.1.1, Def. 1.3.2, Def. 1.3.3]{modsheaf1}).
The composition of morphisms of modulus pairs is given by that of morphisms in $\Sm$.
Thus, we obtain a category $\ulMSm$ of modulus pairs. 

A modulus pair $M$ is \emph{proper} if $\ol{M}$ is proper over $k$. We write $\MSm$ for the full subcategory of $\ulMSm$ which consists of proper modulus pairs. It is our main object of study here.

Recall that the Nisnevich topology on $\Sm$ may be understood by means of a certain \emph{cd-structure} in the sense of Voevodsky \cite{cdstructures}, which is \emph{complete} and \emph{regular} (see loc. cit.).  In \cite{modsheaf1}, it is shown that $\ulMSm$ also admits a complete and regular cd-structure, parallel to the previous one and denoted by $P_{\ulMV}$.
 In \cite{nistopmod}, a more subtle cd-structure $P_{\MV}$ is defined on the category $\MSm$, and shown to be complete and regular as well. We recall in Section \ref{section:cd-structures} the definitions of all these cd-structures. 
 
This paper studies the relationship between the cd-structures $P_{\ulMV}$ and $P_{\MV}$. Our main theorems are too technical to be stated in this introduction;  here they are nevertheless:
\begin{enumerate}
\item Theorem \ref{main:cofinality-MV} (cofinality theorem);
\item Theorem \ref{ex-pc} (existence of partial compactifications).
\end{enumerate}

Let us roughly explain the contents of Theorem \ref{main:cofinality-MV}. Given a complete and regular cd-structure, distinguished squares yield long exact ``Mayer-Vietoris'' sequences for sheaves in the associated topology. Take a distinguished square $S$  in $P_{\ulMV}$. By Proposition \ref{p1}, it may be embedded into a commutative square $T$ in $\MSm$ by a collection of ``compactifications'' (see Definition \ref{d1}). But $T$ has no reason to be in $P_{\MV}$. Theorem \ref{main:cofinality-MV} says that, under a mild normality condition on $S$, one can always lift the embedding $S\inj T$ to an embedding $S\to T'$ with $T'\in P_{\MV}$.

Endow $\ulMSm$ and $\MSm$ with the Grothendieck topologies associated to these cd-structures, and $\Sm$ with the Nisnevich topology. 
Then the following result is a corollary of Theorems \ref{main:cofinality-MV} and \ref{ex-pc}. 
\begin{Th}\label{thm:cocontinuity}
The natural forgetful functors
\begin{align*}
\ul{\omega}_s :& \ulMSm \to \Sm ; \ \ M \mapsto M^\o , \\
\omega_s :& \MSm \to \Sm ; \ \ M \mapsto M^\o 
\intertext{and the left adjoint to $\ulomega_s$}
{\lambda_s:}&{ \Sm\to \ulMSm; \ \ X\mapsto (X,\emptyset)}
\intertext{are continuous and cocontinuous in the sense of \cite[Exp. III]{SGA4}. 
Moreover, the inclusion functor }
\tau_s& : \MSm \to \ulMSm; \ \ M \mapsto M 
\end{align*}

is continuous. 
\end{Th}

(See \S \ref{s:co} for a review of continuity and cocontinuity.)

\begin{Rk} 
On the other hand, $\tau_s$ is not cocontinuous; see Remark \ref{r5.1}. 
Rather, the content of  Theorem  \ref{main:cofinality-MV}  is, morally, that its pro-left adjoint $\tau_s^!$ is continuous for a natural topology on $\pro{}\MSm$ extending that of $\MSm$. However, developing this viewpoint would force us to get into unpleasant categorical and set-theoretic issues, and we prefer to skip it here. 
\end{Rk}

This paper is organised as follows. 
In Section \ref{section:cd-structures}, we recall the definitions of the cd-structure on $\ulMSm$ from  \cite{modsheaf1}, and that on $\MSm$ from \cite{nistopmod}.
Moreover, we state Theorem \ref{main:cofinality-MV}.
In Section \ref{section:pc}, we state and prove Theorem \ref{ex-pc}.
In Section \ref{s3}, we prove Theorem \ref{main:cofinality-MV} in a special case.
In Section \ref{section:cofinality-MV}, we complete the proof of Theorem \ref{main:cofinality-MV}.
In Section \ref{section:cocontinuity}, we prove Theorem \ref{thm:cocontinuity}. The appendices provide technical facts needed in the text.

\subsection*{Acknowledgements} 

The idea of this paper originates in \cite[4.3]{motmod}\footnote{This paper is withdrawn. See \url{https://arxiv.org/pdf/1511.07124v5.pdf}.}.
Many techniques were invented when the second author was staying at IMJ-PRG at Jussieu as a post-doctoral researcher in 2018. 
The second author appreciates the FSMP program which offered him the nice opportunity to work there.
This produced the preprint \cite{KaMi}, which is the direct ancestor to the present paper.

We made many important observations during the first author's stay at RIKEN iTHEMS in September 2019, which enabled us to considerably simplify the arguments of \cite{KaMi} using the off-diagonal functor introduced in \cite{nistopmod}.
We are very grateful for iTHEMS' hospitality and excellent working conditions. 

We thank Shuji Saito for encouraging us to finish this revision by explaining that the techniques of \cite{KaMi} should be a key ingredient to the theory of modulus sheaves with transfers on proper modulus pairs \cite{modsheaf2}.
We also thank Takao Yamazaki for noticing that our results had to do with (co)continuity, which gave us a great insight on what we had been doing in \cite{KaMi}.

\section{Recollection on cd-structures; the cofinality theorem}\label{section:cd-structures}

In this section, we recall definitions of the cd-structures on $\ulMSm$ and $\MSm$ from \cite{modsheaf1} and \cite{nistopmod}. We assume the reader familiar with \cite{cdstructures}, part of whose results is summarised in \cite[A.8]{modsheaf1}.

\subsection{The cd-structure on $\protect\Sm$}
First, recall the Nisnevich cd-structure on $\Sm$.
The following notation is useful.

\begin{defn}
Let $\Sq$ denote the product category $[1]^2 = \{0 \to 1\}^2$. 
For any category $\sC$, define $\sC^\Sq$ as the category of functors $\Sq \to \sC$.
An object of $\sC^\Sq$ is a commutative square in $\sC$, and a morphism of $\sC^\Sq$ is a morphism of commutative squares. 

An object $S \in \sC^\Sq$ will often be depicted as
\begin{equation}\label{eq:S}\begin{gathered}\xymatrix{
S(00) \ar[r]^{v_S} \ar[d]_{q_S} & S(01) \ar[d]^{p_S} \\
S(10) \ar[r]^{u_S} & S(11).
}\end{gathered}\end{equation}
\end{defn}

\begin{defn}
An \emph{elementary Nisnevich square} is an object of $\Sch^\Sq$ of the form 
\[\xymatrix{
W \ar[r] \ar[d] & V \ar[d] \\
U \ar[r] & X
}\]
which satisfies the following properties:
\begin{enumerate}
\item The square is cartesian.
\item The horizontal morphisms are open immersions.
\item The vertical morphisms are \'etale.
\item The morphism $(V - W)_\red \to (X - U)_\red$ is an isomorphism.
\end{enumerate}

Elementary Nisnevich squares whose vertices are in $\Sm$ define a complete and regular cd-structure on $\Sm$ \cite{unstableJPAA}. 
Moreover, the Grothendieck topology associated to the cd-structure is the Nisnevich topology.
\end{defn}

\subsection{The cd-structure on $\protect\ulMSm$}
Next, we recall the definition of $\ulMV$-squares from \cite{modsheaf1}.
We start from the following definition. 

\begin{defn} \ 
\begin{enumerate}
\item A morphism $f : M \to N$ is \emph{ambient} if $f^\o : M^\o \to N^\o$ extends (uniquely) to a morphism $\ol{f} : \ol{M} \to \ol{N}$. 
\item An ambient morphism $f : M \to N$ is \emph{minimal} if $M^\infty = \ol{f}^\ast N^\infty$.
\item Let $\ulMSm^\fin$ (resp. $\MSm^\fin$) be the (non-full) subcategory of $\ulMSm$ (resp. $\MSm$) whose objects are the same as $\ulMSm$ (resp. $\MSm$) and morphisms are ambient morphisms.
\end{enumerate}
\end{defn}

\begin{defn}\label{d1.1}
An $\ulMV^\fin$-square is an object $S \in (\ulMSm^\fin )^\Sq$ of the form \eqref{eq:S} such that 
\begin{enumerate}
\item All morphisms which appear in $S$ are minimal.
\item The square in $\Sch$
\begin{equation*}\begin{gathered}\xymatrix{
\ol{S}(00) \ar[r]^{\ol{v}_S} \ar[d]_{\ol{q}_S} & \ol{S}(01) \ar[d]^{\ol{p}_S} \\
\ol{S}(10) \ar[r]^{\ol{u}_S} & \ol{S}(11).
}\end{gathered}\end{equation*}
\end{enumerate}
is an elementary Nisnevich square.

By \cite[Prop. 3.2.2]{modsheaf1}, the $\ulMVfin$-squares form a complete and regular cd-structure on $\ulMSm^\fin$. 
\end{defn}

\begin{defn}
An $\ulMV$-square is an object $S \in \ulMSm^\Sq$ which belongs to the essential image of the natural (non-full) functor \[(\ulMSm^\fin )^\Sq \to \ulMSm^\Sq.\]

By \cite[Th. 4.1.2]{modsheaf1}, the $\ulMV$-squares form a complete and regular cd-structure on $\ulMSm$, denoted $P_{\ulMV}$. 
\end{defn}

\subsection{The cd-structure on $\protect\MSm$}
Finally, we recall the definition of $\MV$-squares from \cite[\S 4]{nistopmod}.
Recall that 
for any diagram $M_1 \to N \leftarrow M_2$ in $\ulMSm$ (resp. in $\MSm$) such that $M_1^\o \times_{N^\o} M_2^\o \in \Sm$, the fiber product $M_1 \times_N M_2$ is representable in $\ulMSm$ (resp. in $\MSm$)
(see \cite[\S 1.10]{modsheaf1} or \cite[\S 2.2]{nistopmod}) and coproducts (see \cite[Def. 3.1.2]{nistopmod}).

\begin{thm}[off-diagonals; see \protect{\cite[Th. 3.1.3]{nistopmod}}] 
For any morphism $f : M \to N$  in $\ulMSm$ such that $f^\o : M^\o \to N^\o$ is \'etale, there exists a canonical decomposition in $\ulMSm$
\[
M  \times_N M \cong M \sqcup \OD (f),
\]
where $\sqcup$ denotes the coproduct in $\ulMSm$, and the morphism $M \to M  \times_N M$ is the diagonal. 
If $M$ and $N$ are proper modulus pairs, so is $\OD (f)$.

Moreover, if $S$ is a commutative square in $\ulMSm$ of the form \eqref{eq:S} such that 
\begin{enumerate}
\item $u_S^\o$ and $v_S^\o$ are open immersions, and
\item $p_S^\o$ and $q_S^\o$ are \'etale,
\end{enumerate}
then the morphism $v_S \times v_S : S(00) \times_{S(10)} S(00) \to S(01) \times_{S(11)} S(01)$ induces a morphism 
\[
\OD (q_T) \to \OD (p_T)
\]
in $\ulMSm$.
\end{thm}

\begin{defn}[see \protect{\cite[Def. 4.2.1]{nistopmod}}]\label{def:new-MV}
Let $T$ be an object of $\MSm^\Sq$ of the form 
\begin{equation}\label{eq:square-T}\begin{gathered}\xymatrix{
T(00) \ar[r]^{v_T} \ar[d]_{q_T} & T(01) \ar[d]^{p_T} \\
T(10) \ar[r]^{u_T} & T(11).
}\end{gathered}\end{equation}
Then $T$ is called an \emph{$\MV$-square} if the following conditions hold:
\begin{enumerate}
\item \label{thm:new-MV1} $T$ is a pull-back square in $\MSm$.
\item \label{thm:new-MV2} There exist an $\ulMV$-square $S$ such that $S(11) \in \MSm$ and a morphism $S \to T$ in $\ulMSm^\Sq$ such that the induced morphism $S^\o \to T^\o$ is an isomorphism in $\Sm^\Sq$ and $S(11) \to T(11)$ is an isomorphism in $\MSm$. 
In particular, $T^\o$ is an elementary Nisnevich square. 
\item \label{thm:new-MV3} $\OD(q_T) \to \OD (p_T)$ is an isomorphism in $\MSm$.
\end{enumerate}

The $\MV$-squares form a complete and regular cd-structure on $\MSm$, denoted by $P_{\MV}$ (see \cite[Th. 4.3.1, 4.4.1]{nistopmod}). 
\end{defn}

\subsection{A few lemmas} In this subsection, we collect some lemmas which were proven in previous works and will be used repeatedly in the sequel.

\begin{lemma}[\protect{\cite[Lem. 2.2]{KP}}]\label{lKL} 
Let $f:X\to Y$ be a surjective morphism of normal integral schemes, 
and let $D,D'$ be two Cartier divisors on $Y$. 
If $f^*D'\le f^*D$, then $D'\le D$.
\end{lemma}

\begin{lemma}[\protect{\cite[Lem. 3.14]{Mi}}]\label{lem:increasing}
Let $X$ be a quasi-compact scheme and let $D, E$ be Cartier divisors on $X$ with $E \geq 0$.
Assume that the restriction of $D$ to the open subset $X \setminus E \subset X$ is effective.
Then, there exists a natural number $n_0 \geq 1$ such that $D + n \cdot E$ is effective for any $n \geq n_0$.
\end{lemma}

\begin{lemma}[\protect{\cite[Lem. 1.3.7]{modsheaf1}}]\label{lem:no-extra-fiber}
Let $f : X \to Y$ be a separated morphism of schemes, and let $U \subset X$ be an open dense subset. 
Assume that the image $f(U)$ of $U$ is open in $Y$, and the induced morphism $U \to f(U)$ is proper (e.g., an isomorphism).
Then, we have $f^{-1}(f(U)) = U$.
\end{lemma}

\begin{lemma}[\protect{\cite[Lem. 1.6.3]{modsheaf1}}]\label{lem:b-up}
Let $f: U\to X$ be an \'etale morphism of quasi-compact and quasi-separated integral schemes. Let $g: V\to U$ be a proper birational morphism, $T \subset U$ a closed subset such that $g_{|U-T}$ is an isomorphism and  $S$
the closure of $f(T)$ in $X$.
Then there exists a closed subscheme $Z\subset X$ supported in $S$ such that
$U\times_X \Bl_Z(X)  \to U$ factors through $V$. 
\end{lemma}

\subsection{The cofinality theorem}

Recall the following notion from \cite[Def. 1.8.1]{modsheaf1}.

\begin{defn}\label{d1}
For $M \in \ulMSm$, let $\Comp (M)$ be the category whose objects are morphisms $j : M \to N$ in $\ulMSm$ such that 
\begin{enumerate}
\item $N \in \MSm$,
\item $j$ is ambient and minimal,
\item the morphism $\ol{j} : \ol{M} \to \ol{N}$ is a dense open immersion, and 
\item there are effective Cartier divisors $M_N^\infty$ and $C$ on $\ol{N}$ such that $N^\infty = M_N^\infty + C$ and $|C| = \ol{N} - \ol{j}(\ol{M})$,
\end{enumerate}
and morphisms $(j_1 : M \to N_1) \to (j_2 : M \to N_2)$ are morphisms $f : N_1 \to N_2$ in $\MSm$ such that $f \circ j_1 = j_2$.
Note that for any $(j : M \to N) \in \Comp (M)$, the morphism $j^\o : M^\o  \to N^\o$ is an isomorphism in $\Sm$.
By \cite[Lem. 1.8.2]{modsheaf1}, the category $\Comp (M)$ is a cofiltered ordered set.
\end{defn}

\begin{ex} Take $M = (\A^2, 0 \times \A^1)$, $N_1 = (\P^1 \times \P^1, 0 \times \P^1 + \infty \times \P^1 + \P^1 \times \infty)$, 
$N_2 =$ (blow up of $\P^1 \times \P^1$ at $\infty \times \infty$, pullback of $N_1^\infty$).
Then $N_1$ and $N_2$ are both in $\Comp(M)$, and $N_2$ dominates $N_1$.
\end{ex}

We give a similar definition for squares. 

\begin{defn}
Let $S$ be an object in $\ulMSm^\Sq$.
Define $\Comp (S)$ be the category whose objects are morphisms $j : S \to T$ in $\ulMSm^\Sq$ such that for each $(ij) \in \Sq$, the morphism $j(ij) : S(ij) \to T(ij)$ belongs to $\Comp (S(ij))$, and whose morphisms $(j_1 : S \to T_1) \to (j_2 : S \to T_2)$ are morphisms $f : T_1 \to T_2$ in $\MSm^{\Sq}$ such that $f \circ j_1 = j_2$.
\end{defn}

\begin{prop}\label{p1}
For any $S \in \ulMSm^\Sq$, the category $\Comp (S)$ is cofiltered and ordered.
\end{prop}

\begin{proof}
Let $\tau_s : \MSm \to \ulMSm$ be the inclusion functor. 
Then $\tau_s$ admits a pro-left adjoint \cite[I.8.11]{SGA4} $\tau_s^!:\ulMSm\to \pro{}\MSm$ which is represented by $\Comp$, i.e., we have 
\[
\tau_s^! (M) = ``\lim_{(M \to N) \in \Comp (M)}" N
\]
(see \cite[Rem. 1.8.5]{modsheaf1}).
Thus, the assertion follows from Lemma  \ref{l2}, applied to 
$\sC = \MSm$, $\sC' = \ulMSm$, $u=\tau_s$, $v=\tau_s^!$, $I=\Comp$ and $\Delta = \Sq$.
\end{proof}

\begin{defn}
For any $S \in \ulMSm^\Sq$, define $\Comp^{\MV} (S)$ as the full subcategory $\Comp (S)$ consisting of objects $S \to T$ such that $T$ is an $\MV$-square. 
\end{defn}

The main result of this paper is the following.

\begin{thm}\label{main:cofinality-MV}
Let $S$ be an $\ulMV^\fin$-square with $\ol{S}(11)$ normal.
Then, for any $(S \to T) \in \Comp (S)$, there exists $(S \to T') \in \Comp^{\MV} (S)$ which dominates $(S \to T)$ in $\Comp (S)$, and such that $T'(11) \to T(11)$ is ambient and minimal (hence an isomorphism in $\MSm$). 

In particular, $\Comp^{\MV} (S)$ is cofinal in $\Comp (S)$.
\end{thm}

The proof of Theorem \ref{main:cofinality-MV} will be given in \S\S \ref{s3} and \ref{section:cofinality-MV}.

\section{Partial compactifications}\label{section:pc}

\subsection{Definition and statement}

\begin{definition}\label{def:partial-comp} Let $S$ be an $\ulMVfin$-square.

\begin{enumerate}
\item 
We say that $S$ is \emph{normal} if $\ol{S}(11)$ is normal (Note that this implies that $\ol{S}(ij)$ is normal for all $i,j \in \{0,1\}$). 
\item 
A \emph{partial compactification} of $S$ is a morphism $S \to S'$ in \allowbreak $(\ulMSm^\fin )^\Sq$ such that 
\begin{enumerate}
\item $S'$ is an $\ulMVfin$-square, 
\item the morphism $S(11) \to S'(11)$ belongs to $\Comp (S(11))$,
\item $\ol{S}(ij) \to \ol{S}'(ij)$ are open immersions, and
\item $\ol{S}(ij) \iso \ol{S}'(ij) \times_{\ol{S}'(11)} \ol{S}(11)$.
\end{enumerate}
\item An $\ulMVfin$-square is called \emph{partially compact} if $S(11) \in \MSm$.
\end{enumerate}
\end{definition}

\begin{ex}
The simplest case is when $S$ is given by a Nisnevich square of schemes with empty divisors (and this is the essential case).

Take (in characteristic $\neq 2$)
\begin{align*}
S(11) &= (\A^1 - \{0\}, \emptyset)\\
S(10) &= (\A^1 - \{0,1\}, \emptyset)\\
S(01) &= (\A^1 - \{0, -1\} , \emptyset)\\
S(00) &= (\A^1 - \{0,1,-1\}, \emptyset)
\end{align*}
with $\ol{S}(01) \to \ol{S}(11)$ the square map $t \mapsto t^2$ (and the horizontal maps the inclusions). Then $S$ is an $\ulMV^\fin$-square,  as it is a distinguished Nisnevich square if we forget the empty divisor. 
A partial compactification $S \to S'$ is given by 
\begin{align*}
S'(11) &= (\P^1, 0 + \infty)\\
S'(10) &= (\P^1 - \{1\}, 0+\infty)\\
S'(01) &= (\P^1 - \{0,-1\}, 2\infty)\\
S'(00) &= (\P^1 - \{0,1,-1\}, 2\infty)
\end{align*}
(Note that the last two $\infty$'s have multiplicity $2$.)
\end{ex}

\begin{remarks}
(1) If $S \to S'$ is a partial compactification, then the isomorphism in Condition (d) induces an isomorphism $S (ij) \iso S'(ij) \times_{S'(11)} S(11)$ in $\ulMSm^\fin$, where the right hand side denotes the fiber product in $\ulMSm^\fin$, which exists by the minimality of the projection maps \cite[Cor. 1.10.7]{modsheaf1}.

(2) If $S \to S'$ is a partial compactification, then the induced morphism $S(ij)^\o \to S'(ij)^\o$ are isomorphisms for all $i,j \in \{0,1\}$.
This is true for $i=j=1$ by Condition (2) (b). For other $i,j$, we need to prove $\ol{S}'(ij)-\ol{S}(ij) \subset |S' (ij)^\infty |$. 
The left hand side equals the pullback of $\ol{S}'(11) - \ol{S}(11)$ by the map $\ol{S}'(ij) \to \ol{S}'(11)$ by Condition (2) (d).
Since $\ol{S}'(11) - \ol{S}(11) \subset |S'(11)^\infty|$ by the previous case, we obtain $\ol{S}'(ij)-\ol{S}(ij) \subset |S'(11)^\infty  \times_{\ol{S}'(11)} \ol{S}'(ij)| = |S'(ij)^\infty |$, where the last equality follows from the minimality of $S'(ij) \to S(11)$. 
\end{remarks}

The main result of this section is the following theorem.

\begin{thm}\label{ex-pc}
For any $\ulMVfin$-square $S$ and for any compactification $T\in \Comp (S(11))$, there exists a partial compactification $S \to S'$ such that $S'(11) \in \Comp (S(11))$ dominates $T$, and the morphism $S'(11) \to T$ is minimal. 
\end{thm}

The proof of Theorem \ref{ex-pc} will be given in the following subsections.

\subsection{The Zariski case}\label{subsection:Zariski}

Before going into the proof of the general case, we will describe the proof in the case that $S$ is a Zariski square, i.e., that $\ol{p}_S : \ol{S}(01) \to \ol{S}(11)$ is an open immersion. 
This subsection is only used in the sequel as a guide for the reader.

Take any object $(S(11) \to T) \in \Comp (S(11))$, and set $Z_1 := \ol{S}(11) - \ol{S}(10)$ and $Z_2 := \ol{S}(01) - \ol{S}(00)$.
Let $\ol{Z}_i$ be the closure of $Z_i$ in $\ol{T}$ for $i=1,2$.

\subsubsection{Special case}\label{case1} If $\ol{Z}_1 \cap \ol{Z}_2$ is empty, we set $\ol{S}'(10) := \ol{T} - \ol{Z}_1$, $\ol{S}'(01) := \ol{T} - \ol{Z}_2$ and $\ol{S}'(00) := \ol{S}'(10) \cap \ol{S}'(01)$.
Moreover, set $S'(11) := T$, $S'(ij)^\infty := S'(11)^\infty \cap \ol{S}'(ij)$ and $S'(ij):=(\ol{S}'(11),S'(11)^\infty)$ for $(ij) \neq (11)$.
Then we obtain a partial compactification $S \to S'$, where the maps $S(ij) \to S'(ij)$ are induced by natural open immersions. 

\subsubsection{General case}\label{case2} In general, let $\pi : \ol{T}_1 \to \ol{T}$ be the blow-up of $\ol{T}$ along $\ol{Z}_1 \times_{\ol{T}} \ol{Z}_2$.
Then the closure of $Z_1$ in $\ol{T}_1$ and the closure of $Z_2$ in $\ol{T}_1$ do not intersect. 
Therefore, by applying the above construction by replacing $T$ with $T_1 := (\ol{T}_1, \pi^\ast T^\infty )$, we obtain a partial compactification of $S$.

The general case of Theorem \ref{ex-pc} follows this strategy, with rather substantial complications.

\subsection{A general construction}\label{subsection:general-observation}

In this subsection, we make a preliminary construction for the proof of the general case. 
Set $Z_1 := \ol{S}(11) - \ol{S}(10)$ and $Z'_1 := \ol{S}(01) - \ol{S}(00)$. 
Since $\ol{S}$ is an elementary Nisnevich square, 
the natural morphism $Z'_1 \to Z_1$ is an isomorphism, and we have $Z'_1 \cong Z_1 \times_{\ol{S}(11)} \ol{S}(01)$.

Contrary to the Zariski case, we cannot regard $\ol{S}(01)$ and $\ol{S}(00)$ as open subsets of $\ol{T}$.
So, we take a compactification $\ol{S}(01) \to \ol{R}$ such that $\ol{p}_S : \ol{S}(01) \to \ol{S}(11)$ extends to a morphism $\ol{p} : \ol{R} \to \ol{T}$ of schemes over $k$
\footnote{For example, take a compactification $\ol{S}(01)\to \ol{R}_0$, and define $\ol{R}$ as the graph of the rational map $\ol{R}_0 \tto \ol{T}$.}, and set $R := (\ol{R},R^\infty) := (\ol{R},\ol{p}^\ast T^\infty)$.
Thus we obtain a minimal morphism $p : R \to T$.

In the Zariski case, we considered the closures of $Z_1$ and $Z_2$ in $\ol{T}$ and studied their intersection.
In the general case, we will consider closures in $\ol{R}$.

We need the following elementary observation. 
Consider the open subscheme $U:=\ol{p}^{-1} (\ol{S}(11))$ of $\ol{R}$.
Then we have the following commutative diagram
\[\xymatrix{
Z'_1 \ar[r] \ar[d]_{\wr} \ar@{}[rd]|\square & \ol{S}(01) \ar[r] \ar[d]_{\ol{p}_S} & U \ar[ld] \ar[r] & \ol{R} \ar[d]^{\ol{p}} \\
Z_1 \ar[r] & \ol{S}(11) \ar[rr] & & \ol{T},
}\]
where we regard $Z'_1$ and $Z_1$ as reduced closed subschemes. 

\begin{lemma}\label{lem:Z123}
\ 
\begin{enumerate}
\item
The inclusion $Z'_1 \subset U$ is a closed immersion. 

\item $Z'_1 = \ol{p}^{-1} (Z_1) \cap \ol{S}(01)$.

\item Regard
\[
\ol{p}^{-1} (Z_1) := Z_1 \times_{\ol{T}} \ol{R},
\]
as a closed subscheme of $U$.
Then there exists an open and closed subscheme $Z_3$ of $U$ such that $\ol{p}^{-1} (Z_1) = Z'_1 \sqcup Z_3$.

\item Set $Z_2 := U - (Z_3 \sqcup \ol{S}(01))$. Then $Z_2$ is a closed subset of $U$. We endow $Z_2$ with the reduced scheme structure. 

\item The closed subschemes $Z'_1, Z_2, Z_3$ of $U$ are disjoint from each other. 
\end{enumerate}
\end{lemma}

\begin{proof}
We prove (1).
The composite $Z'_1 \to Z_1 \to \ol{S}(11)$ is a proper morphism since $Z'_1 \to Z_1$ is an isomorphism. 
Since it factors through $\ol{U}$ and since $\ol{U} \to \ol{T}$ is separated, we conclude that $Z'_1 \to \ol{U}$ is proper, hence a closed immersion. 

(2) follows from the isomorphism $Z'_1 \cong Z_1 \times_{\ol{S}(11)} \ol{S}(01)$.
 
We prove (3).
(1) implies that $Z'_1$ is a closed subscheme of $\ol{p}^{-1} (Z_1)$.
On the other hand, $Z'_1$ is open also in $\ol{p}^{-1} (Z_1)$ by (2). 
Therefore, taking $Z_3 := \ol{p}^{-1} (Z_1) - Z'_1$, we finish the proof.

(4) immediately follows from (3). 

(5) By construction, we have $U-\ol{S}(01) = Z_2 \sqcup Z_3$ and $Z'_1 \subset \ol{S}(01)$.
This finishes the proof.
\end{proof}

\begin{remark}\label{rem:Z3}
In the Zariski case, we have $Z_3 = \emptyset$. 
\end{remark}

Let $\ol{Z}_1$ be the closure of $Z_1$ in $\ol{T}$.  
Moreover, let $\ol{Z}'_1$, $\ol{Z}_2$, $\ol{Z}_3$ be the closures of $Z'_1, Z_2, Z_3$ in $\ol{R}$, endowed with their reduced scheme structures.  

\begin{lemma}\label{lem:cleanness-V}
Set $V:=\ol{R} - (\ol{Z}_2 \cup \ol{Z}_3)$.
Then we have
\begin{enumerate}
\item $U\cap V = \ol{S}(01)$;
\item $q^{-1}(\ol{S}(11)) = \ol{S}(01)$,  where $q$ is the composite $V \to \ol{R} \xrightarrow{\ol{p}} \ol{T}$.
\end{enumerate}
\end{lemma}

\begin{proof}
We have $Z_2\sqcup Z_3=U-\ol{S}(01)\subset \ol{R}-\ol{S}(01)$ and $\ol{R}-\ol{S}(01)$ is closed in $\ol{R}$, hence $\ol{Z}_2\cup \ol{Z}_3\subset \ol{R}-\ol{S}(01)$, so $
\ol{S}(01)\subset V$ and $U\cap V\supseteq \ol{S}(01)$. But $U-\ol{S}(01)=Z_2\sqcup Z_3$ and $V\cap (Z_2\sqcup Z_3)=\emptyset$, hence we have equality in (1).
Finally, $q^{-1}(\ol{S}(11)) =U\cap V$ so (1) $\iff$ (2).
\end{proof}

\subsection{Proof of Theorem \protect\ref{ex-pc} in a special case}\label{subsection:pc-special} In the Zariski case, this subsection reduces to \S\ \ref{case1} (see Remark \ref{rem:Zariski} below).

Let $S$, $(S(11) \to T) \in \Comp (S(11))$ and $\ol{S}(01) \to \ol{R}$ be as in the previous subsection. 
Moreover, define $Z_1, Z'_1, Z_2 , Z_3$ and $\ol{Z}_1, \ol{Z}'_1$, $\ol{Z}_2$, $\ol{Z}_3$ in the same way as before (see Lemma \ref{lem:Z123}). 

In this subsection, we assume the following condition on $T$ and $\ol{R}$.
\begin{itemize}
\item[$(\ast)_{T,\ol{R}}$] \ Let $V,q$ be as in Lemma \ref{lem:cleanness-V}.
Let $V_\flat \subset V$ be the flat locus of the composite $q : V \subset \ol{R} \xrightarrow{\ol{p}} \ol{T}$. Then $V_\flat$ contains $\ol{Z}'_1$.
\end{itemize}

\begin{remark}\label{rem:Zariski}
Assume that $S$ is a Zariski square, i.e., that $\ol{p}_S$ is an open immersion, and take $\ol{R}$ to be $\ol{T}$.
Then the condition $(\ast)_{T,\ol{T}}$ is equivalent to that $\ol{Z}_1 \cap \ol{Z}_2 = \emptyset$.
Indeed, by Remark \ref{rem:Z3}, we have $V_\flat = V = \ol{T} - \ol{Z}_2$.
Moreover, we have $\ol{Z}_1 = \ol{Z}'_1$.
Therefore $\ol{Z}'_1 \subset V_\flat \iff \ol{Z}_1 \cap \ol{Z}_2 = \emptyset$.
\end{remark}

The general case will be treated in the next subsection.

Let $j : V_\et \subset V$ be the \'etale locus of $q$.
Define 
\begin{align*}
S'(11) &:= T, \\
S'(10) &:= (\ol{S}'(10), S'(10)^\infty ) := (\ol{T} - \ol{Z}_1, T^\infty \cap (\ol{T} - \ol{Z}_1)), \\ 
S'(01) &:= (\ol{S}'(01), S'(01)^\infty ) = (V_\et , j^\ast q^\ast T^\infty ).
\end{align*}
Then the open immersion $\ol{T} - \ol{Z}_1 \to \ol{T}$ and the morphism $q \circ j : V_\et \to \ol{T}$ induce minimal morphisms $S'(10) \to S'(11)$ and $S'(01) \to S'(11)$.

Set $S'(00) := S'(10) \times_{S'(11)} S'(01)$ be the fiber product in $\ulMSm^\fin$, which exists by the minimality of (one of) the projections \cite[Cor. 1.10.7]{modsheaf1}. 
In our situation, we have 
\[
S'(00) = (\ol{S}'(10) \times_{\ol{S}'(11)} \ol{S}'(01), \text{the pullback of $S'(11)^\infty$}).
\]
Thus, we obtain a pull-back diagram 
\[S': 
\begin{gathered}\xymatrix{
S'(00) \ar[r] \ar[d] & S'(01) \ar[d] \\
S'(10) \ar[r] & S'(11)
}\end{gathered}\]
in $\ulMSm^\fin$.
By construction, for each $(ij) \in \Sq$, we have $\ol{S}(ij) \subset \ol{S}'(ij)$. 
Moreover, the open immersions induce minimal morphisms $S(ij) \to S'(ij)$. 
Therefore, we obtain a morphism $S \to S'$ in $(\ulMSm^\fin )^\Sq$.

\begin{prop}\label{prop:easy-pc}
The morphism $S \to S'$ is a partial compactification of $S$.
\end{prop}

We need the following two lemmas for the proof.

\begin{lemma}\label{claim:generic-iso}
In the factorization 
\begin{equation}\label{toto}
Z'_1 \to q_\flat^{-1} (Z_1)\to Z_1
\end{equation}
both morphisms are isomorphisms. 
\end{lemma}

\begin{proof} 
We have the following commutative diagram:
\[\xymatrix{
Z'_1 \ar[r] \ar[d]_{\wr} \ar@{}[rd]|\square & \ol{S}(01) \ar[r] \ar[d] \ar@{}[rd]|\square &  V \ar[d]^{q} \\
Z_1 \ar[r] & \ol{S}(11) \ar[r] & \ol{T},
}\]
where the left square is cartesian since $S$ is an $\ulMVfin$-square, and the right square is also cartesian thanks to Lemma \ref{lem:cleanness-V}. 
Since $ \ol{S}(01) \subset V_\flat$, this implies that the following commutative square
\[\xymatrix{
Z'_1 \ar[r] \ar[d]_{\wr}  & V_\flat \ar[d]^{q_\flat} \\
Z_1 \ar[r] & \ol{T}
}\]
 is  also cartesian. So the first morphism of \eqref{toto} is an isomorphism, and hence so is the second one. 
This concludes the proof.
\end{proof}

For the next lemma, recall that we have $\ol{Z}'_1 \subset V_\flat$ by assumption.
From now on, we regard $\ol{Z}'_1$ and $\ol{Z}_1$ as reduced closed subschemes. 

\begin{lemma}\label{lem:iso-ZZ'}
The morphism $\ol{Z}'_1 \to \ol{Z}_1$
 is an isomorphism.
Moreover, the induced morphism $\ol{Z}'_1 \to q_\flat^{-1} (\ol{Z}_1) := \ol{Z}_1 \times_{\ol{T}} V_\flat$ is an isomorphism. 
\end{lemma}

\begin{proof}
Let $q_{\flat ,Z}: q_\flat^{-1} (\ol{Z}_1) \to \ol{Z}_1$ be the base change of $q : V_\flat \to \ol{T}$ by the closed immersion $\ol{Z}_1 \subset \ol{T}$.
Then we obtain the following commutative diagram.

\[\xymatrix{
Z'_1 \ar[r] \ar[d] & \ol{Z}'_1 \ar[r] \ar[d] & q_\flat^{-1} (\ol{Z}_1) \ar[r] \ar[ld]^{q_{\flat,Z}} & V_\flat \ar[r] \ar[rd]_{q_\flat} &  V \ar[d]^{q} \\
Z_1 \ar[r] & \ol{Z}_1 \ar[rrr] &&& \ol{T}.
}\]

\begin{claim}\label{claim:3.4.5}
 The morphism $q_{\flat,Z}$ is an isomorphism.
\end{claim}

\begin{proof}
Since $q_\flat : V_\flat \to \ol{T}$ is flat, so is $q_{\flat,Z}$ by base change.
Moreover, $q_{\flat ,Z}$ is an isomorphism over the dense open subset $Z_1 \subset \ol{Z}'_1$ by Lemma \ref{claim:generic-iso}.
Therefore, $q_{\flat ,Z}$ is an open immersion by Theorem \ref{strong-lemma}.

On the other hand, note that the morphism $\ol{Z}'_1 \to \ol{Z}_1$ decomposes as $\ol{Z}'_1 \subset q_{\flat ,Z}^{-1} (\ol{Z}_1) \xrightarrow{q_{\flat, Z}} \ol{Z}_1$.
Since $\ol{Z}'_1 \to \ol{Z}_1$ is dominant and proper, it is surjective. 
Therefore, the open immersion $q_{\flat ,Z}$ is indeed an isomorphism.
This finishes the proof of Claim \ref{claim:3.4.5}.
\end{proof}

Note that $\ol{Z}'_1 \to \ol{Z}_1$ is surjective. 
By Claim \ref{claim:3.4.5}, this implies that the closed immersion $\ol{Z}'_1 \to q_\flat^{-1} (\ol{Z}_1)$ is also surjective. 
Since $\ol{Z}'_1$ is reduced by construction, and since $q_\flat^{-1} (\ol{Z}_1) \cong \ol{Z}_1$ is also reduced as $\ol{Z}_1$ is reduced by construction, the surjection $\ol{Z}'_1 \to q_\flat^{-1} (\ol{Z}_1)$ must be an isomorphism of schemes.  
This finishes the proof of Lemma \ref{lem:iso-ZZ'}.
\end{proof}

\begin{proof}[Proof of Proposition \ref{prop:easy-pc}]
We will check Conditions (a)-(d) in Definition \ref{def:partial-comp}.
Conditions (b) and (c) are satisfied by construction. 

We check (d). The case $(ij)=(11)$ is obvious. 
The case $(ij)=(10)$ can be checked by 
\[
\ol{S}'(10) \cap \ol{S}(11) = (\ol{S}'(11) - \ol{Z}_1) \cap \ol{S}(11) = \ol{S}(11) - Z_1 = \ol{S}(10).
\]
The case $(ij) = (01)$ follows from Lemma \ref{lem:cleanness-V}.
The case $(ij) = (00)$ follows from $\ol{S}(00) \cong \ol{S}(10) \times_{\ol{S}(11)} \ol{S}(01)$.

Finally, we check Condition (a), i.e., that $S'$ is an $\ulMV^\fin$-square. 
Since all edges of $S'$ are minimal, it suffices to show that the square $\ol{S}'$ of schemes is an elementary Nisnevich square. 
The horizontal maps of $\ol{S}'$ are open immersions, and the vertical maps of $\ol{S}'$ are \'etale by construction. 
In view of Lemma \ref{lem:iso-ZZ'}, noting that $\ol{S}'(11) - \ol{S}(10) = \ol{Z}_1$, it suffices to prove the following claim. 

\begin{claim}
$\ol{Z}'_1 \subset \ol{S}'(01) = V_\et$.
\end{claim}

\begin{proof}
Since $\ol{Z}'_1 \to \ol{Z}_1$ is an isomorphism by Lemma \ref{lem:iso-ZZ'}, the flat morphism $q_\flat : V_\flat \to \ol{T}$ is unramified at each point of $\ol{Z}'_1$. This shows that $q_\flat$ is \'etale at each point of $\ol{Z}'_1$ by \cite[Th. 17.6.1]{EGA4-4}. 
This finishes the proof of the claim. 
\end{proof}

Thus, we have finished the proof of Proposition \ref{prop:easy-pc}.
\end{proof}

\subsection{A refinement of the general construction}

\begin{prop}\label{prop:pc-separation}
In \S \ref{subsection:general-observation}, we may choose $\ol{S}(01) \to \ol{R}$ satisfying the following conditions.
\begin{enumerate}
\item $\ol{p}_S : \ol{S}(01) \to \ol{S}(11)$ extends to a morphism $\ol{p} : \ol{R} \to \ol{T}$.
\item $\ol{Z}'_1 \cap \ol{Z}_2 = \emptyset$ and $\ol{Z}'_1 \cap \ol{Z}_3 = \emptyset$, where $\ol{Z}'_1$ and $\ol{Z}_2$ are the closures of $Z'_1$ and $Z_2$ in $\ol{R}$.
\end{enumerate}
\end{prop}

Before going into the proof, we prepare a definition and a lemma which will be used several times.

\begin{defn}\label{def:blow-up-mod}
Let $M \in \ulMSm$, and let $F$ be a closed subscheme of $\ol{M}$ such that $F \cap M^\o = \emptyset$.
Let 
\[
\ol{\pi} : \Bl_F (\ol{M}) \to \ol{M}
\]
be the blow-up of $\ol{M}$ along $F$, and let 
\[
\ol{\nu} : \Bl_F (\ol{M})^N \to \Bl_F (\ol{M})
\]
be the normalization of $\Bl_F (\ol{M})^N$.
Set 
\begin{align*}
\Bl_F (M) &:= (\Bl_F (\ol{M}) , \ol{\pi}^\ast M^\infty ), \\
\Bl_F (M)^N &:= (\Bl_F (\ol{M})^N , \ol{\nu}^\ast \ol{\pi}^\ast M^\infty ).
\end{align*}

By construction, $\Bl_F (M)^\o = (\Bl_F (M)^N)^\o = M^\o$.
Moreover, the maps $\ol{\pi}$ and $\ol{\nu}$ induce minimal morphisms $\pi : \Bl_F (M) \to M$ and $\nu : \Bl_F (M)^N \to \Bl_F (M)$. 

We call $\pi : \Bl_F (M) \to M$ (resp. $\pi \nu : \Bl_F (M)^N \to M$) \emph{the blow-up of $M$ along $F$} (resp. \emph{the normalized blow-up of $M$ along $F$}). 
\end{defn}

\begin{lemma}\label{lem:blow-up-lift}
Let $M_0 \in \ulMSm$ and $(M_0 \to M) \in \Comp (M_0)$. 
Let $F$ be a closed subscheme of $\ol{M}$ with $F \cap M^\o = \emptyset$.
Assume that $F \cap \ol{M}_0$ is an effective Cartier divisor on $\ol{M}_0$. 
Then the following assertions hold.
\begin{enumerate}
\item The open immersion $\ol{M}_0 \to \ol{M}$ lifts uniquely to an open immersion $\ol{M}_0 \to \Bl_F (\ol{M})$. 
This defines an object $(M_0 \to \Bl_F (M)) \in \Comp (M_0)$ which dominates $M_0 \to M$.
\item If $\ol{M}_0$ is normal, the open immersion $\ol{M}_0 \to \ol{M}$ lifts uniquely to an open immersion $\ol{M}_0 \to \Bl_F (\ol{M})^N$.
This defines an object $(M_0 \to \Bl_F (M)^N ) \in \Comp (M_0)$ which dominates $M_0 \to M$.
\end{enumerate}
\end{lemma}

\begin{proof}
Since $(M_0 \to M) \in \Comp (M_0)$, there exist effective Cartier divisors $M_{0,M}^\infty ,C_M$ on $\ol{M}$ such that $M^\infty = M_{0,M}^\infty + C_M$, $|C_M| = \ol{M} - \ol{M}_0$ and $M_{0,M}^\infty \cap \ol{M}_0 = M_0^\infty$.

(1): The assumption shows that $\pi : \Bl_F (\ol{M}) \to \ol{M}$ is an isomorphism over $\ol{M}_0$.
Therefore, the open immersion $j : \ol{M}_0 \to \ol{M}$ lifts uniquely to an open immersion $j_1 : \ol{M}_0 \to \Bl_F (\ol{M})$:
\[\xymatrix{
& \Bl_F (\ol{M}) \ar[d]^\pi \\
\ol{M}_0 \ar[r]^j \ar[ur]^{j_1} & \ol{M}.
}\]

By construction, we have
\[
\Bl_F (M)^\infty = \pi^\ast M^\infty = \pi^\ast M_{0,M}^\infty + \pi^\ast C_{M}.
\]
Moreover, we have 
\[
j_1^\ast \pi^\ast M_{0,M}^\infty = j^\ast M_{0,M}^\infty = M_0^\infty ,
\]
and
\begin{align*}
|\pi^\ast C_{M}| &= \pi^{-1} (|C_{M}|) = \pi^{-1} (\ol{M} - j(\ol{M}_0)) = \Bl_F (\ol{M}) - \pi^{-1} (j(\ol{M}_0)) \\
&=\Bl_F (\ol{M}) - \pi^{-1} (\pi j_1 (\ol{M}_0)) \\
&= \Bl_F (\ol{M}) - j_1 (\ol{M}_0),
\end{align*}
where the last equality follows from the Lemma \ref{lem:no-extra-fiber}. This proves that $j_1$ defines an object $(M_0 \to \Bl_F (M)) \in \Comp (M_0)$ which dominates $M_0 \to M$.

(2): The assumption shows that $\pi \nu : \Bl_F (\ol{M})^N \to \ol{M}$ is an isomorphism over $\ol{M}_0$.
Therefore, the open immersion $j : \ol{M}_0 \to \ol{M}$ lifts uniquely to an open immersion $j_2 : \ol{M}_0 \to \Bl_F (\ol{M})^N$.
The rest of the argument is the same as above. This finishes the proof.
\end{proof}

\begin{proof}[Proof of Proposition \ref{prop:pc-separation}]
We start from a construction as in \S \ref{subsection:general-observation}; for clarity, we write $\ol{R}_0$ instead of $\ol{R}$ but keep the other notation $(Z'_1 ,\ol{Z}'_1, \dots)$.
Let
\[
\pi_1 : \ol{R}_1 := \Bl_{\ol{Z}'_1 \times_{\ol{R}_0} \ol{Z}_2} (\ol{R}_0) \to \ol{R}_0
\]
be the blow-up of $\ol{R}_0$ along $\ol{Z}'_1 \times_{\ol{R}} \ol{Z}_2$.
Then $\pi_1$ is an isomorphism over the open subscheme $\ol{S}(01) \subset \ol{R}_0$ since $\ol{Z}'_1 \times_{\ol{R}} \ol{Z}_2 \cap \ol{S}(01) \subset Z'_1 \cap Z_2 = \emptyset$ by  Lemma \ref{lem:Z123} (5).
Therefore, the open immersion $\ol{S}(01) \to \ol{R}_0$ lifts uniquely to an open immersion $\ol{S}(01) \to \ol{R}_1$.
Moreover, the strict transforms $\pi_1^\# \ol{Z}'_1$ and $\pi_1^\# \ol{Z}_2$, i.e., the closures of $Z'_1$ and $Z_2$ in $\ol{R}_1$, do not intersect. 

Similarly, let 
 \[
\pi_2 : \ol{R}_2 := \Bl_{\pi_1^\# \ol{Z}'_1 \times_{\ol{R}_1} \pi_1^\# \ol{Z}_3} (\ol{R}_1) \to \ol{R}_1
\]
be the blow-up of $\ol{R}_1$ along $\pi_1^\# \ol{Z}'_1 \times_{\ol{R}_1} \pi_1^\# \ol{Z}_3$.
Then $\pi_2$ is an isomorphism over the open subscheme $\ol{S}(01) \to \ol{R}_1$ since $\pi_1^\# \ol{Z}'_1 \times_{\ol{R}_1} \pi_1^\# \ol{Z}_3 \cap \ol{S}(01) \cong \ol{Z}'_1 \times_{\ol{R}} \ol{Z}_3 \cap \ol{S}(01) \subset Z'_1 \cap Z_3 = \emptyset$ by  Lemma \ref{lem:Z123}.
Therefore, the open immersion $\ol{S}(01) \to \ol{R}_1$ lifts uniquely to an open immersion $\ol{S}(01) \to \ol{R}_2$. 
Moreover, the strict transforms $\pi_2^\# \pi_1^\# \ol{Z}'_1$ and $\pi_2^\# \pi_1^\# \ol{Z}'_3$, i.e., the closures of $Z'_1$ and $Z_3$ in $\ol{R}_2$, do not intersect. 
Recalling that $\pi_1^\# \ol{Z}'_1 \cap  \pi_1^\# \ol{Z}'_2 = \emptyset$, we also see that the strict transforms $\pi_2^\# \pi_1^\# \ol{Z}'_1$ and $\pi_2^\# \pi_1^\# \ol{Z}'_2$, i.e., the closures of $\ol{Z}'_1$ and $\ol{Z}_2$ in $\ol{R}_2$, do not intersect.
Thus, setting $\ol{R} := \ol{R}_3$, we get the desired compactification $\ol{S}(01) \to \ol{R}$.
\end{proof}

\subsection{End of proof of Theorem \protect\ref{ex-pc}}
It suffices to show the following.

\begin{prop}
There exist $(S(11) \to T') \in \Comp (S(11))$ and a compactification $\ol{S}(01) \to \ol{R}_1$ which satisfy Condition $(\ast)_{T',\ol{R}_1}$ from \S \ref{subsection:pc-special}, such that $S(11) \to T'$ dominates $S(11) \to T$ and the morphism $T' \to T$ is minimal. 
\end{prop}

\begin{proof}
Take a compactification $\ol{S}(01) \to \ol{R}$ as in Proposition \ref{prop:pc-separation}.

Recall that $V=\ol{R} - (\ol{Z}_2 \cup \ol{Z}_3)$.
By assumption, we have $\ol{Z}'_1 \subset V$.
Moreover, we have $q^{-1} (\ol{S}(11)) = \ol{S}(01)$ by Lemma \ref{lem:cleanness-V}.
In particular, $q : V \to \ol{T}$ is \'etale over $\ol{S}(11) \subset \ol{T}$.

Therefore, by the theorem of ``platification'' \cite[Th. 5.2.2]{rg} of Raynaud-Gruson, there exists a closed subscheme $C \subset \ol{T} - \ol{S}(11)$ which satisfies the following condition: 
define $T_1 := \Bl_C (T)$ (see Definition \ref{def:blow-up-mod}).
Let 
\[
\pi'_1 : \ol{R}_1 := \pi_1^\# \ol{R} \to \ol{R}
\]be the strict transform, i.e., the blow-up of $\ol{R}$ along the closed subscheme $\ol{p}^{-1} (C) \subset \ol{R}$.
Let $\ol{p}_1 := \pi_1^\# \ol{p} : \ol{R}_1 \to \ol{T}_1$ be the lift of $\ol{p}$.
Then the strict transform 
\[
V_1 := \pi_1^\# (V) = V \times_{\ol{T}} \ol{T}_1 \subset \ol{R}_1
\]
of $V$ is flat over $\ol{T}_1$.

Note that $\pi_1$ is an isomorphism over $\ol{S}(11)$ since $C \cap \ol{S}(11) = \emptyset$.
Therefore, Lemma \ref{lem:blow-up-lift} shows that the open immersion $\ol{S}(11) \to \ol{T}$ lifts uniquely to an open immersion $\ol{S}(11) \to \ol{T}_1$, and it defines an object $(S(11) \to T_1) \in \Comp (S(11))$.

Moreover, $\pi'_1 : \ol{R}_1 := \pi_1^\# \ol{R} \to \ol{R}$ is an isomorphism over $\ol{S}(01) \subset \ol{R}$ since $\ol{S}(01)$ lies over $\ol{S}(11)$.
Therefore, the open immersion $\ol{S}(01) \to \ol{R}$ lifts uniquely to an open immersion $\ol{S}(01) \to \ol{R}_1$.

Since $\ol{Z}'_1 \cap \ol{Z}_2 = \ol{Z}'_1 \cap \ol{Z}_3 = \emptyset$, we have the corresponding equality for their strict transforms: $\pi_1^\# \ol{Z}'_1 \cap \pi_1^\#  \ol{Z}_2 = \pi_1^\#  \ol{Z}'_1 \cap \pi_1^\#  \ol{Z}_3 = \emptyset$.
Since $\pi_1^\# \ol{Z}'_1, \pi_1^\# \ol{Z}_2, \pi_1^\# \ol{Z}_3$ are the closures of $Z'_1 , Z_2 ,Z_3$ in $\ol{R}_1$, the pair $(S(11) \to T_1) \in \Comp (S(11))$ and $\ol{S}(01) \to \ol{R}_1$ satisfies $(\ast)_{T_1,\ol{R}_1}$.
Moreover, the blow-up $\pi_1 : \ol{T}_1 \to \ol{T}$ induces a minimal morphism $T_1 \to T$ by construction. 
This finishes the proof.
\end{proof}

\begin{cor}\label{cor:pc-normal}
If $S$ is a normal $\ulMVfin$-square and if $S \xrightarrow{j} S'$ is a partial compactification, then there exists a partial compactification $S \xrightarrow{j_1} S'_1$ such that $S'_1$ is normal and a minimal morphism $S'_1 \xrightarrow{p} S$ such that $pj_1=j$. 
\end{cor}

\begin{proof} 
Let $S'_1 (ij) = \Bl_{\emptyset } (S' (ij))^N$ be the normalized blow-up along the empty subscheme for $(ij) \in \Sq$.
More explicitely, we have
\[
S'_1 (ij) = (\ol{S}'(ij)^N, S'(ij)^\infty \times_{\ol{S}'(ij)} \ol{S}'(ij)^N).
\]

By Lemma \ref{lem:blow-up-lift}, the morphisms $S \to S'(ij)$ uniquely lift to $(S(ij) \to S'_1(ij)) \in \Comp (S(ij))$ for all $(ij) \in \Sq$.

Since $\ol{S}'(ij) \to \ol{S}'(11)$ is \'etale for each $(ij) \in \Sq$, we have
\[
\ol{S}'(ij)^N \cong \ol{S}'(ij) \times_{\ol{S}'(11)} \ol{S}'(11)^N,
\]
hence
\[
S'_1 (ij) \cong S'(ij) \times_{S'(11)} S'(11)^N
\]
where the right hand side denotes the fiber product in $\ulMSm^\fin$, which exists by the minimality of (one of) the projection maps ((\cite[Cor. 1.10.7]{modsheaf1})).
Therefore, $S'_1(ij)$'s form an $\ulMVfin$-square $S'_1$. 
This finishes the proof.  
\end{proof}

\section{Cofinality of $\MV$-squares: the partially compact case}\label{s3}

\subsection{Main result} In this section we prove the following special case of Theorem \ref{main:cofinality-MV}:

\begin{thm}\label{main:cofinality-MV-partial}
Theorem \ref{main:cofinality-MV} is true if $S$ is partially compact in the sense of Definition \ref{def:partial-comp} (3). 
\end{thm}

This is the technical heart of the paper. The strategy is simple: we achieve successively conditions (1) and (3) of Definition \ref{def:new-MV}, Condition (2) being automatic.

In the sequel, we fix a normal $\ulMV^\fin$-square $S$; it will be assumed partially compact from Subsection \ref{s3.4} onwards.

\subsection{Another general construction}\label{sssection:cof-gencon}

Here, we prepare a general setting which will be used in the proof of Theorem \ref{main:cofinality-MV-partial}.

Let  $(S \to T) \in \Comp (S)$. Since $(S(ij) \to T(ij)) \in \Comp (S(ij))$ for each $(ij) \in \Sq$, we can find an effective Cartier divisor $D_{ij}$ on $\ol{T}(ij)$ such that $|D_{ij}| = \ol{T}(ij) - \ol{S}(ij)$.

Write 
\[\xymatrix{
T(00) \ar[r]^{v_T} \ar[d]_{q_T} \ar[dr]^{d_T} & T(01) \ar[d]^{p_T} \\
T(10) \ar[r]^{u_T} & T(11),
}\]
where $d_{T} := p_T v_T = u_T q_T$.
All morphisms in the diagram are ambient by assumption. 

We shall need the following Condition \ref{condition:ans}.
We recall a definition and a lemma.

\begin{defn}\label{defn:uni-sup}
Two effective Cartier divisors $D, E$ on a scheme $X$ \emph{have a universal supremum} if $D \times_X E$ is an effective Cartier divisor on $X$.
If $D,E$ have a universal supremum, we define an effective Cartier divisor  $\sup (D,E)$ on $X$ by $\sup (D,E) :=  D+E-D \times_X E$. 
\end{defn}

\begin{lemma}
Let $D,E$ be effective Cartier divisors on a scheme $X$ which have a universal supremum.
Then, for any morphism $f:Y \to X$ in $\Sch$ such that $Y$ is normal and such that $f(T) \not\subset |D| \cup |E|$ for any irreducible component $T$ of $Y$, the effective Cartier divisors  $f^\ast D$ and $f^\ast E$ have a universal supremum, and  $f^\ast \sup (D,E) = \sup (f^\ast D,f^\ast E)$.
Moreover, if we regard $\sup (f^\ast D,f^\ast E)$ as a Weil divisor on the normal scheme $Y$, it coincides with the smallest Weil divisor which is larger than $f^\ast D$ and $f^\ast E$.
\end{lemma}

\begin{proof}
See \cite[Def. 1.10.2, Rem. 1.10.3]{modsheaf1}  or \cite[Lem. 2.2.1]{nistopmod}.
\end{proof}

\begin{condition}\label{condition:ans}
\ 
\begin{enumerate}
\item $T$ is ambient, i.e., $T \in (\MSm^\fin )^\Sq$.
\item $\ol{T}(ij)$ is normal for each $(ij) \in \Sq$.
\item $\ol{q}_T^\ast D_{10}$ and $\ol{v}_T^\ast D_{01}$ have a universal supremum on $\ol{T}(00)$.
\end{enumerate}
\end{condition}

\begin{lemma}\label{lem:normal-T}
For any $(S \to T_0) \in \Comp (S)$, there exists $(S \to T) \in \Comp (S)$ which dominates $(S \to T_0)$ and such that $T$ satisfies Conditions \ref{condition:ans}.
\end{lemma}

\begin{proof}
By the graph trick, there exists $(S \to T_1 ) \in \Comp (S)$ which dominates $(S \to T_0) \in \Comp (S)$, such that $T_1 \in (\MSm^\fin )^\Sq$.

Set $F := \ol{q}_{T_1}^\ast D_{10} \times_{\ol{T}(00)} \ol{v}_{T_1}^\ast D_{01}$, and 
let
\[
T(00) := \Bl_F (T_1 (00))^N \to T_1 (00)
\]
be the normalized blow up of $T_1 (00)$ along $F$ (see Definition \ref{def:blow-up-mod}).
Since $F \cap \ol{S}(00) = \emptyset$ by construction, and since $\ol{S}(00)$ is normal by assumption, Lemma \ref{lem:blow-up-lift} (2) shows that the open immersion $\ol{S}(00) \to \ol{T}_1 (00)$ lifts uniquely to an open immersion $\ol{S}(00) \to \ol{T}(00)$, which defines an object $(S(00) \to T(00)) \in \Comp (S(11))$. 

For $(ij) \in \Sq - \{(00)\}$, define
\[
T(ij) := (\ol{T}_1 (ij)^N, T(ij)^\infty \times_{\ol{T}(ij)} \ol{T}(ij)^N).
\]

Then, for each $(ij) \in \Sq - \{(00)\}$, the morphism $S(ij) \to T_1 (ij)$ lifts uniquely to an object $S(ij) \to T(ij)$ by Lemma \ref{lem:blow-up-lift} (2), noting that the normalization is the normalized blow-up along the closed subset $\emptyset$.
Moreover, for each $(ij) \to (kl)$ in $\Sq$, the morphism $\ol{T}_1 (ij) \to \ol{T}_1 (kl)$ lifts to a morphism $\ol{T}(ij) \to \ol{T}(kl)$ by the universal property of normalization. 
Therefore, $T$ is ambient. 

The other conditions in Conditions \ref{condition:ans} hold by construction.
 \end{proof}

In the rest of this subsection, we fix $(S\to T)\in \Comp(S)$ verifying Conditions \ref{condition:ans}. Set
\[
D := \sup (\ol{q}_T^\ast D_{10}, \ol{v}_T^\ast D_{01}),
\]
where the notation on the right hand side is from Definition \ref{defn:uni-sup}.
Then $D$ is an effective Cartier divisor on $\ol{T}(00)$ by assumption. 

\begin{lemma}\label{lem:foc}
$|D| = |D_{00}| = \ol{T}(00) - \ol{S}(00)$.
\end{lemma}

\begin{proof}
The second equality is by definition. 
To show the first one, taking the complements of both sides, we are reduced to proving $\ol{S}(00) = \ol{q}_T^{-1} (\ol{S}(10)) \cap \ol{v}_T^{-1} (\ol{S}(01))$. 
The inclusion $\ol{S}(00) \subset \ol{q}_T^{-1} (\ol{S}(10)) \cap \ol{v}_T^{-1} (\ol{S}(01))$ is obvious.
By the universal property of fiber product, there exists a unique morphism $\ol{q}_T^{-1} (\ol{S}(10)) \cap \ol{v}_T^{-1} (\ol{S}(01)) \to \ol{S}(10) \times_{\ol{S}(11)} \ol{S}(01) = \ol{S}(00)$ which is compatible with $\ol{q}_S$ and $\ol{v}_S$. 
We can check that this map is inverse to the inclusion map, by restricting them to the dense open subset $S^\o (00) = T^\o (00)$. This finishes the proof. 
\end{proof}

\begin{lemma}\label{lem:bound-DD}
There exists a positive integer $n_T$ such that for any $n \geq n_T$, we have 
\begin{equation*}\label{eq:bound-DD}
T(00)^\infty \leq \ol{d}^\ast T(11)^\infty + n D.
\end{equation*}
\end{lemma}

\begin{proof}
We have $T(00)^\infty |_{\ol{S}(00)} = \ol{d}^\ast T(11)^\infty |_{\ol{S}(00)} = S(00)^\infty$ by the minimality of $S(00) \to T(00)$ and $S(00) \to S(11) \cong T(11)$.
Therefore, in view of Lemma \ref{lem:foc} and Lemma \ref{lem:increasing}, we finish the proof.
\end{proof}

For any non-negative integers $m$ and $n$, define $T_{m,n} \in (\MSm^\fin )^\Sq$ by
\begin{align*}
T_{m,n}(11) &:= T(11) \\
T_{m,n} (10) &:= (\ol{T}(10), T(10)^\infty + mD_{10})  \\
T_{m,n} (01) &:= (\ol{T}(01), T(01)^\infty + nD_{01})  \\
T_{m,n}(00) &:= T_{m,n} (10) \ctimes_{T_{m,n}(11)} T_{m,n} (01),
\end{align*}
where $\ctimes$ denotes the canonical model of fiber product introduced in \cite[Def. 2.2.2]{nistopmod}.
Note that for each $(ij) \in \Sq - \{(00)\}$, the open immersion $\ol{S}(ij) \to \ol{T}(ij)$ induces an object $(S(ij) \to T_{m,n}(ij)) \in \Comp (S(ij))$.
Moreover, it is easy to see by construction of canonical model of fiber product that the morphism $S(00) \to T_{m,n}(00)$ in $\ulMSm$, which is induced by the universal property of fiber product, is ambient and minimal, and defines an object $(S(00) \to T_{m,n}(00)) \in \Comp (S(00))$.

Therefore, we obtain $(S \to T_{m,n}) \in \Comp (S)$ for any $m,n$.
Write 
\[\xymatrix{
T_{m,n}(00) \ar[r]^{v_{T_{m,n}}} \ar[d]_{q_{T_{m,n}}} \ar[dr]^{d_{T_{m,n}}} & T_{m,n}(01) \ar[d]^{p_{T_{m,n}}} \\
T_{m,n}(10) \ar[r]^{u_{T_{m,n}}} & T_{m,n}(11).
}\]

Note that $T_{m,n}$ is cartesian in $\MSm$ for any $m,n$ by construction. 

\subsection{Cofinality of cartesian squares}

\begin{prop}\label{prop:cofinality-pullback} Let $(S\to T)\in \Comp(S)$ verifying Conditions \ref{condition:ans}. Let $n_T$ be as in Lemma \ref{lem:bound-DD}, and $T_{m,n}$ as constructed above. Then for any integers $m,n \geq n_T$, there exists a morphism $(S \to T_{m,n}) \to (S \to T)$ in $\Comp (S)$.
\end{prop}

\begin{proof}
For $(ij) \in \Sq - \{(00)\}$ and for any $m,n$, there exists a natural morphism $T_{m,n} (ij) \to T(ij)$ in $\ulMSm^\fin$ by construction. 

Our task is to show that the isomorphism $T_{m,n}(00)^\o \to T(00)^\o$ in $\Sm$ defines a morphism $T_{m,n}(00) \to T(00)$ in $\MSm$ for $m,n \geq n_T$.

Let $\Gamma$ be the graph of the rational map $\ol{T}_{m,n} (00) \dashrightarrow \ol{T}(00)$, and $\Gamma^N \to \Gamma$ the normalization of $\Gamma$. 
Then we obtain the following commutative diagrams of schemes:
\[\xymatrix{
\ol{T}_{m,n} (00) \ar[d]_{\ol{q}_{T_{m,n}}} & \ar[l]_(0.35)a \Gamma^N \ar[r]^(0.4)b &\ol{T}(00) \ar[d]^{\ol{q}_T} \\
\ol{T}_{m,n} (10) \ar[rr]^{=} &  &\ol{T}(10),
}\ \ 
\xymatrix{
\ol{T}_{m,n} (00) \ar[d]_{\ol{v}_{T_{m,n}}} & \ar[l]_(0.35)a \Gamma^N \ar[r]^(0.4)b &\ol{T}(00) \ar[d]^{\ol{v}_T} \\
\ol{T}_{m,n} (01) \ar[rr]^{=} &  &\ol{T}(01).
}
\]

\begin{claim}\label{claim:DDD}
$\sup (a^\ast \ol{q}_{T_{m,n}}^\ast D_{10}, a^\ast \ol{v}_{T_{m,n}}^\ast D_{01}) = b^\ast D$.
\end{claim}

\begin{proof}
We have $\ol{q}_T b = \ol{q}_{T_{m,n}}a$ and $\ol{v}_T b = \ol{v}_{T_{m,n}}a$ by the commutativity of the above diagrams.
Therefore, 
\begin{align*}
\sup (a^\ast \ol{q}_{T_{m,n}}^\ast D_{10}, a^\ast \ol{v}_{T_{m,n}}^\ast D_{01})) 
&= \sup (b^\ast \ol{q}_{T}^\ast D_{10}, b^\ast \ol{v}_{T}^\ast D_{01})) \\
&= b^\ast \sup (\ol{q}_{T}^\ast D_{10}, \ol{v}_{T}^\ast D_{01})) \\
&= b^\ast D.
\end{align*}
This finishes the proof.
\end{proof}

By construction of $T_{m,n}(00)$, we have 
\[
T_{m,n} (00)^\infty = \sup (\ol{q}_{T_{m,n}}^\ast T_{m,n}(10)^\infty, \ol{v}_{T_{m,n}}^\ast T_{m,n}(01)^\infty),
\]
hence
\begin{equation}\label{eq:3.1.1}
a^\ast T_{m,n} (00)^\infty = \sup (a^\ast \ol{q}_{T_{m,n}}^\ast T_{m,n}(10)^\infty, a^\ast \ol{v}_{T_{m,n}}^\ast T_{m,n}(01)^\infty).
\end{equation}
By construction of $T_{m,n}$ and by choice of $m,n$, we have 
\begin{align*}
\ol{u}_{T_{m,n}}^\ast T(11)^\infty + n_T D_{10}  &\leq \ol{u}_{T_{m,n}}^\ast T(11)^\infty + mD_{10} \leq T_{m,n}(10)^\infty ,\\
\ol{p}_{T_{m,n}}^\ast T(11)^\infty + n_T D_{01} &\leq \ol{p}_{T_{m,n}}^\ast T(11)^\infty + nD_{01} \leq T_{m,n}(01)^\infty .
\end{align*}
Combining these with \eqref{eq:3.1.1}, we have
\begin{align*}
a^\ast T_{m,n} (00)^\infty 
&\geq a^\ast \ol{d}_{T_{m,n}}^\ast T(11)^\infty + n_T \sup (a^\ast \ol{q}_{T_{m,n}}^\ast D_{10}, a^\ast \ol{v}_{T_{m,n}}^\ast D_{01}) \\
&= b^\ast (\ol{d}_{T}^\ast T(11)^\infty + n_T D),
\end{align*}
where the last equality follows from Claim \ref{claim:DDD} and $\ol{d}_{T_{m,n}} a = \ol{d}_T b$.
Therefore, in view of Lemma \ref{lem:bound-DD}, we have
\[
a^\ast T_{m,n} (00)^\infty \geq b^\ast T(00)^\infty ,
\]
which shows that the isomorphism $T_{m,n}(00)^\o \to T(00)^\o$ in $\Sm$ defines a morphism $T_{m,n}(00) \to T(00)$ in $\MSm$. This finishes the proof of Proposition \ref{prop:cofinality-pullback}.
\end{proof}

\begin{cor}\label{c3.1} The subset 
\[
\{(S \to T) \in \Comp (S) \ | \ \text{$T \in (\MSm^\fin)^\Sq$,  $T$ is cartesian in $\MSm$}\}
\]
is cofinal in $\Comp (S)$.\qed
\end{cor}

\subsection{Topological study of a certain diagram}\label{s3.4}
By Corollary \ref{c3.1}, we may and do assume in Theorem \ref{main:cofinality-MV-partial} that $T \in (\MSm^\fin)^\Sq$ and that $T$ is cartesian in $\MSm$.
Since $T^\o$ is an elementary Nisnevich square, the morphism $\OD (q_T)^\o \to \OD (p_T)^\o$ is an isomorphism in $\Sm$, and it induces an admissible morphism $\OD (q_T) \to \OD (p_T)$.
Our task is to modify $T$ in order to make this morphism invertible in $\MSm$. In this subsection, we prepare the ground to show that we only need to increase multiplicities of divisors, which will be done in the next subsection.

We take the notation of \S \ref{sssection:cof-gencon}.
Let $\Gamma$ be the graph of the birational map $\ol{\OD (q_{T})} \dashrightarrow \ol{\OD (p_{T})}$.
Let $\nu : \Gamma \to \ol{\OD (p_{T})}$ be the natural map, and let $s_{i} : \ol{\OD (p_{T})} \to \ol{T}(01)$ be the natural $i$-th projection for $i=1,2$.

Set $H:=\nu^\ast s_1^\ast D_{01} \times_{\Gamma} \nu^\ast s_2^\ast D_{01}$, and let
\begin{equation}\label{eq:gamma-1}
\pi : \Gamma_1 := \Bl_{H} (\Gamma)^N \to \Gamma
\end{equation}
be the normalized blow-up of $\Gamma$ along $H$, and set $b:=\nu \pi$. 
Then $b^\ast s_1^\ast D_{01}$ and $b^\ast s_2^\ast D_{01}$ have a universal supremum (see Definition \ref{defn:uni-sup}) by construction. 
Set 
\begin{equation}\label{eq:D-prime}
D' := \sup (b^\ast s_1^\ast D_{01}, b^\ast s_2^\ast D_{01} ).
\end{equation}

By \cite[Prop. 2.2.7 (3), Lem. 4.1.4]{nistopmod}, there are natural open immersions $j_q : \ol{\OD (q_{S})} \to \ol{\OD (q_{T})}$ and $j_p : \ol{\OD (p_{S})} \to \ol{\OD (p_{T})}$.
They lift uniquely to open immersions $j'_q : \ol{\OD (q_{S})} \to \Gamma_1$ and $j'_p : \ol{\OD (p_{S})} \to \Gamma_1$, since $\ol{\OD (q_{S})}$ and $\ol{\OD (p_{S})}$ are normal by construction and since $H \cap j_p (\ol{\OD (p_S)}) = H \cap j_q (\ol{\OD (q_S)}) = \emptyset$.

Thus we obtain commutative diagrams of schemes ($i=1,2$)
\begin{equation}\label{eq:diagrams1}\begin{gathered}\xymatrix{
\ol{\OD (q_{S})} \ar[rrrr]^{\sim} \ar[d]_{j_q} \ar[rrd]^{j'_q} & & & & \ar[lld]_{j'_p} \ol{\OD (p_{S})} \ar[d]^{j_p} \\
\ol{\OD (q_{T})} \ar[d]_{t_i} && \ar[ll]_(0.4)a \Gamma_1 \ar[rr]^(0.4)b && \ol{\OD (p_{T})} \ar[d]^{s_i} \\
\ol{T} (00) \ar[rrrr]^{\ol{v}_T} &&&& \ol{T} (01),
}\end{gathered}\end{equation}
 where $t_i$ is the natural $i$-th projection, and $a$ is the composite $\Gamma_1 \to \Gamma \to \ol{\OD (q_{T})}$.

Set 
\begin{equation}\label{eq:U}
U := j'_q (\ol{\OD (q_{S})}) =  j'_p (\ol{\OD (p_{S})}).
\end{equation}

Consider the composite
\[
c :  \ol{\OD (p_T )} \to \ol{T(01) \ctimes_{T(11)} T(01)} \to \ol{T}(11),
\]
where the first morphism is the natural inclusion, and the second is the structural morphism.

\begin{lemma}\label{claim:3.1.9}
$a^\ast \OD (q_{T})^\infty |_U = b^\ast \OD (p_{T})^\infty |_U = b^\ast c^\ast T(11)^\infty |_U$.
\end{lemma}

\begin{proof}
The natural morphism $\OD (q_S) \to \OD (p_S)$ is an isomorphism in $\ulMSm^\fin$ by \cite[Prop. 4.1.5]{nistopmod}.
Moreover, the morphisms $\OD (q_S) \to \OD (q_{T})$ and $\OD (p_S) \to \OD (p_{T})$ are minimal. 
Therefore, the first equality follows from the commutativity of \eqref{eq:diagrams1}. 
The second assertion follows from the minimality of the natural morphism $\OD (p_S) \to S(10) \ctimes_{S(11)} S(01) \to S(11)  \to T(11)$.
This finishes the proof.
\end{proof}

\begin{lemma}\label{claim:3.1.10}
$\ol{\OD (q_{T})} - j_q (\ol{\OD (q_S)}) = t_1^{-1} \ol{v}_T^{-1} (|D_{01}|) \cup t_2^{-1} \ol{v}_T^{-1} (|D_{01}|)$.
\end{lemma}

\begin{proof}
Set \[A := t_1^{-1} \ol{v}_T^{-1} (\ol{S}(01)) \cap t_2^{-1} \ol{v}_T^{-1} (\ol{S}(01)).\]

Then the assertion is equivalent to the equality
\[j_q (\ol{\OD (q_S)}) =  A.\]

The inclusion $\subset$ follows from the commutativity of \eqref{eq:diagrams1}. 
By the universal property of the fiber product, we have a natural morphism 
\[
\gamma : A \to \ol{S}(01) \times_{\ol{S}(11)} \ol{S}(01).
\]

\begin{claim}\label{claim:4.2.10}
$\gamma (A) \cap \Delta (\ol{S}(01)) = \emptyset$.
\end{claim}

\begin{proof}
First, note that $\ol{\OD (q_{T})}$ is the closure of 
\[
\OD (q_T^\o) = \OD (q_S^\o) = S(00)^\o \times_{S(10)^\o} S(00)^\o - \Delta (S(00)^\o)
\] 
by construction (see the proof of \cite[Th. 3.1.3]{nistopmod}).
Therefore, since $A \subset \ol{\OD (q_T)}$ by construction, the open immersion 
\[
\OD (q_T^\o) \to A
\]
is also dense. 
In particular, $\OD (q_T^\o)$ is dense in $\gamma (A)$.

Since $\OD (q_T^\o) \cap \Delta (S(01)^\o) = \emptyset$ and 
since $\Delta (S(01)^\o)$ is dense in  $\Delta (\ol{S}(01))$, we have $\OD (q_T^\o) \cap \Delta (\ol{S}(01)) = \emptyset$.
Therefore, since $\Delta (\ol{S}(01))$ is closed in $\ol{S}(01) \times_{\ol{S}(11)} \ol{S}(01)$, we have
\[ 
\gamma (A) \cap \Delta (\ol{S}(01)) \subset \ol{\OD (q_T^\o)} \cap \Delta (\ol{S}(01)) = \emptyset.
\]
This finishes the proof of Claim \ref{claim:4.2.10}.
\end{proof}

By Claim \ref{claim:4.2.10}, the morphism $\gamma$ induces
\[
A \to j_q (\ol{\OD (q_S)}).
\]

This map is inverse to the inclusion morphism, since their restrictions to the dense interior $S^\o (01) \times_{S^\o(11)} S^\o (01)$ are the identity. 
Therefore, $j_q (\ol{\OD (q_S)}) = A$.
This finishes the proof of Lemma \ref{claim:3.1.10}.
\end{proof}

\begin{lemma}\label{claim:UD} We have
$\Gamma_1 - U = |D'|$ (see \eqref{eq:D-prime} and \eqref{eq:U} for the definitions of $D'$ and $U$).
\end{lemma}

\begin{proof}
By applying $a^{-1}$ to both sides of the equality in Lemma \ref{claim:3.1.10}, by using the commutativity of the above diagram and by Lemma \ref{lem:no-extra-fiber}, we have
\begin{align*}
\Gamma_1 - U = b^{-1} s_1^{-1} (|D_{01}|) \cup b^{-1} s_2^{-1} (|D_{01}|) = |D'|,
\end{align*}
where the last equality follows from the construction of $D'$.
This finishes the proof.
\end{proof}

\subsection{Increasing multiplicities} By Lemmas \ref{lem:increasing}, \ref{claim:3.1.9} and \ref{claim:UD}, there exists a positive integer $n$ such that 
\begin{equation}\label{eq:nD}
a^\ast \OD (q_{T})^\infty \leq b^\ast c^\ast T(11)^\infty + nD' .
\end{equation}

Set 
\[
T' := T_{0,n},
\]
where the right hand side is defined as in \S \ref{sssection:cof-gencon}.

Then there exists a natural morphism $T' \to T$ in $(\MSm^\fin )^\Sq$,
and we obtain the following commutative diagram 
\begin{equation}\label{eq:cartesian-od}\begin{gathered}\xymatrix{
\OD (q_{T'}) \ar[r]  \ar[d] & \OD (q_{T}) \ar[d] \\
\OD (p_{T'}) \ar[r] & \OD (p_{T})
}\end{gathered}\end{equation}
in $\MSm$ by the functorial property of $\OD$ (see \cite[Th. 3.1.3]{nistopmod}).

\begin{lemma}\label{lem:cartesian-OD}
Diagram \eqref{eq:cartesian-od} is cartesian in $\MSm$.
\end{lemma}

\begin{proof}
Consider the following commutative diagram:
\[\xymatrix{
\OD (q_{T'}) \ar[r]  \ar[d] & \OD (q_{T}) \ar[d]  \ar[r]^{\pi_{q}}& T(10) \ar[d]^{u_T} \\
\OD (p_{T'}) \ar[r] & \OD (p_{T}) \ar[r]^{\pi_{p}} & T(11),
}\]
in $\MSm$, where $\pi_p$ and $\pi_q$ are the natural projections to the bases of fiber products. 
Since $T$ is cartesian, the right square is cartesian by \cite[Prop. 3.1.4]{nistopmod}.
Since $T'$ is also cartesian and since $T'(11)=T(11)$ and $T'(10) = T(10)$ by construction, the large square is cartesian by {\it loc. cit.}.
Therefore, a general argument shows that the left square is cartesian. 
This finishes the proof.
\end{proof}

The main point of this subsection is:

\begin{prop}\label{prop:OD-T-prime}
The natural morphism $\OD (q_{T'}) \to \OD (p_{T'})$ is an isomorphism in $\MSm$.
\end{prop}

\begin{proof}
We will construct an inverse morphism. 
By Lemma \ref{lem:cartesian-OD},
it suffices to show that $\OD (p_{T'}) \to \OD (p_T)$ lifts to a morphism $\OD (p_{T'}) \to \OD (q_T)$.

Let $\Gamma_2$ be the graph of the rational map $\ol{\OD (p_{T'})} \dashrightarrow \Gamma_1$.
Then we obtain the following commutative diagrams ($i=1,2$)
\[\xymatrix{
\ol{\OD (p_{T'})} \ar[d]_{s'_i} \ar@/_30pt/[dd]_{c'} & \ar[l]_(0.35){a'} \Gamma_2^N  \ar[r]^{b'}&  \Gamma_1 \ar[r]^(0.35){b} & \ol{\OD (p_T)} \ar[d]^{s_i} \ar@/^30pt/[dd]^{c} \\
\ol{T}'(01) \ar[rrr]^{=} \ar[d]_{\ol{p}_{T'}} & &  & \ol{T}(01) \ar[d]^{\ol{p}_T} \\
\ol{T}'(11) \ar[rrr]^{=} & & &  \ol{T}(11),
}\]
where $s'_i$ are natural projections, and $c':=\ol{p}_{T'} s'_1 = \ol{p}_{T'} s'_2$.

By \eqref{eq:nD}, 
\[
(b')^\ast a^\ast \OD (q_{T})^\infty \leq (b')^\ast (b^\ast c^\ast T(11)^\infty + nD'). 
\]

The commutativity of the above diagram shows 
\begin{align*}
 (b')^\ast b^\ast c^\ast T(11)^\infty =  (a')^\ast  (c')^\ast T(11)^\infty .
\end{align*}

By \eqref{eq:D-prime}, 
\begin{align*}
(b')^\ast D' &=  \sup ((b')^\ast b^\ast s_1^\ast D_{01},(b')^\ast b^\ast s_2^\ast D_{01} ) \\
&= \sup ((a')^\ast (s'_1)^\ast  D_{01}, (a')^\ast (s'_2)^\ast D_{01} ) .
\end{align*}

Combining these, we get
\begin{align*}
(b')^\ast a^\ast \OD (q_{T})^\infty \leq \sup_{i=1,2} ((a')^\ast  ((c')^\ast T(11)^\infty + n (s'_i)^\ast  D_{01})).
\end{align*}

By the admissibility of $p_{T} : T(01) \to T(11)$, and noting that $p_{T}=p_{T'}$ by construction, we have
\[
(c')^\ast T(11)^\infty = (s'_i)^\ast \ol{p}_{T'}^\ast T(11)^\infty \leq (s'_i)^\ast T(01)^\infty .
\]
for each $i=1,2$

Therefore, we have
\begin{equation}\label{eq:4.2.8}
(b')^\ast a^\ast \OD (q_{T})^\infty 
\leq 
\sup_{i=1,2} ((a')^\ast (s'_i)^\ast (T(01)^\infty + n D_{01})).
\end{equation}

Recall that $T'(01)^\infty = T(01)^\infty + nD_{01}$ by the definition of $T'=T_{0,n}$.
Moreover, by the construction of $\OD$, we have $\OD (p_{T'})^\infty = \sup_{i=1,2} ((s'_i)^\ast T'(01)^\infty )$ (see the proof of \cite[Th. 3.1.3]{nistopmod}).
Therefore, we have

\begin{align*}
\sup_{i=1,2} ((a')^\ast (s'_i)^\ast (T(01)^\infty + n D_{01})) 
&= \sup_{i=1,2} ((a')^\ast (s'_i)^\ast T'(01)^\infty ) \\
&= (a')^\ast  \sup_{i=1,2} ((s'_i)^\ast T'(01)^\infty ) \\
&= (a')^\ast \OD (p_{T'})^\infty .
\end{align*}

Combined with \eqref{eq:4.2.8}, this implies 
\begin{equation}\label{eq:4.2.9}
(b')^\ast a^\ast \OD (q_{T})^\infty 
\leq
(a')^\ast \OD (p_{T'})^\infty .
\end{equation}

Let $\Gamma_3$ be the graph of the rational map $\ol{\OD (p_{T'})} \dashrightarrow \ol{\OD (q_T)}$, and $\Gamma_3^N \to \Gamma_3$ be the normalization of $\Gamma_3$.
Recalling that $\Gamma_2$ is the graph of the rational map $\ol{\OD (p_{T'})} \dashrightarrow \Gamma_1$, we have a natural birational morphism $\Gamma_2 \to \Gamma_3$.
Thus, we have the following commutative diagram:
\[\xymatrix{
\ol{\OD (p_{T'})} \ar[d]_{\id} & \ar[l]_(0.35){a'} \ar[r]^{b'} \ar[d]^{\xi} \Gamma_2^N & \Gamma_1 \ar[d]^a \\
\ol{\OD (p_{T'})} & \ar[l]_(0.35)l \ar[r]^(0.35)r \Gamma_3^N & \ol{\OD (q_T)} ,
}\]
where $\xi$ is induced from the morphism $\Gamma_2 \to \Gamma_3$ by the universal property of normalization.

By \eqref{eq:4.2.9} and by the commutativity of the above diagram, we have
\[
\xi^\ast r^\ast \OD (q_{T})^\infty 
\leq
\xi^\ast l^\ast \OD (p_{T'})^\infty .
\]

Since $\xi$ is a proper birational by construction, Lemma \ref{lKL} and \eqref{eq:4.2.9} imply
\[
r^\ast \OD (q_{T})^\infty 
\leq
l^\ast \OD (p_{T'})^\infty ,
\]
which shows that $\OD (p_{T'}) \to \OD (q_{T})$ is defined in $\MSm$.
This concludes the proof of Proposition \ref{prop:OD-T-prime}.
\end{proof}

\subsection{End of proof of Theorem \protect\ref{main:cofinality-MV-partial}}

We now assume that $S$ is partially compact.
Then $T'$ satisfies Condition (2) in Definition \ref{def:new-MV}. 
Since $T'$ is cartesian in $\MSm$, it also satisfies Condition (1).
Moreover, it satisfies Condition (3) by Proposition \ref{prop:OD-T-prime}. 
Therefore, $T'$ is an $\MV$-square, and hence $T' \in \Comp^{\MV} (S)$. This finishes the proof.

\section{Cofinality of $\protect\MV$-squares: the general case}\label{section:cofinality-MV}

In this section, we complete the proof of Theorem \ref{main:cofinality-MV}.

\subsection{Preparatory lemmas}

\begin{lemma}\label{lem:pc-modif}
Let $S$ be a $\ulMVfin$-square and let $S \to S'$ be a partial compactification. 
Let $F$ be a closed subscheme of $\ol{S}'(11)$ such that $F \cap S'(11)^\o = \emptyset$, and set $F_{ij} := F \times_{\ol{S}'(11)} \ol{S}'(ij) \subset \ol{S}'(ij)$ for $(ij) \in \Sq$. Note that $F_{11}=F$ by definition. 

Let 
\[S'_F (ij):= \Bl_{F_{ij}} S'(ij)\] be the blow-up of $S'(ij)$ along $F_{ij}$ (see Definition \ref{def:blow-up-mod}).
Note that for any morphism  $(ij) \to (kl)$ in $\Sq$, the structure morphism $S'(ij) \to S'(kl)$ lifts to a morphism $S'_F(ij) \to S'_F(kl)$ in $\ulMSm^\fin$ by the universal property of blowing up.

Then the resulting square $S'_F$ is an $\ulMVfin$-square. 
Moreover, the morphism $S \to S'$ lifts uniquely to a morphism $S \to S'_F$ in $(\ulMSm^\fin )^\Sq$, and it is a partial compactification of $S$.
\end{lemma}

\begin{proof}
Since $\ol{S}(ij) \to \ol{S}(11)$ are \'etale (hence flat), we have 
\begin{equation}\label{eq:3.1.1a}
\ol{S}'_F (ij) \cong \ol{S}'(ij) \times_{\ol{S}'(11)} \ol{S}'_F(11).
\end{equation}

Therefore, the resulting square $\ol{S}'_F$ is an elementary Nisnevich square as base change. 
Since $S'_F(ij)^\infty$ are the pull-back of $S'_F(11)^\infty$ by construction, we obtain an $\ulMVfin$-square $S'_F$.

By Lemma \ref{lem:blow-up-lift}, the object $(S(11) \to S'(11)) \in \Comp (S(11))$  lifts uniquely to $(S(11) \to S'_F (11)) \in \Comp (S(11))$.
Similarly, for $(ij) \in \Sq - \{(11)\}$, the open immersion $\ol{S}(ij) \to \ol{S}'(ij)$  lifts uniquely to an open immersion $\ol{S}(ij) \to \ol{S}'_F(ij)$ since $F_{ij} \cap \ol{S}(ij)$ is an effective Cartier divisor by assumption.
Therefore, we obtain a morphism $S \to S'_F$ in $\ulMSm^\fin$. 

The last thing to check is Condition (c) in Definition \ref{def:partial-comp}.
Since $S \to S'$ is a partial compactification, the natural morphism  
\[
\ol{S}(ij) \xrightarrow{\sim} \ol{S}'(ij) \times_{\ol{S}'(11)} \ol{S}(11).
\]
is an isomorphism.
Combining this with \eqref{eq:3.1.1a} as above, for each $(ij) \in \Sq$, we have
\begin{align*}
\ol{S}'_F (ij) \times_{\ol{S}'_F(11)} \ol{S}(11) 
&\cong \ol{S}'(ij) \times_{\ol{S}'(11)} \ol{S}'_F(11) \times_{\ol{S}'_F(11)} \ol{S}(11) \\
&\cong \ol{S}'(ij) \times_{\ol{S}'(11)} \ol{S}(11) \\
&\cong \ol{S}(ij).
\end{align*}
This finishes the proof of Lemma \ref{lem:pc-modif}.
\end{proof}

\begin{lemma}\label{lem:comp-modif}
Let $f : M \to N$ be a morphism in $\ulMSm$ with $N \in \MSm$.
Assume that $f$ is ambient and minimal, and that $f^\o : M^\o \to N^\o$ is an isomorphism. 
Then there exists an isomorphism $N \xrightarrow{\sim} N'$ in $\MSm$ such that the composite $M \to N \xrightarrow{\sim} N'$ belongs to $\Comp (M)$.
\end{lemma}

\begin{proof}
Take any compactification $j:\ol{M} \to X$ and let $\Gamma_1$ be the graph of the birational map $g:X \dashrightarrow \ol{N}$.
Since the composite $gj=\ol{f}$ is a morphism of schemes, the open immersion $j$ lifts uniquely to an open immersion $j_1 : \ol{M} \to \Gamma_1$.
Let 
\[
\pi : \Gamma_2 := \Bl_{(\Gamma_1 - j_1 (\ol{M}))_\red } (\Gamma_1) \to \Gamma_1
\]
be the blow-up of $\Gamma_1$ along the closed subscheme $(\Gamma_1 - j_1 (\ol{M}))_\red$.
Then $j_1$ lifts to an open immersion $j_2 : \ol{M} \to \Gamma_2$ since $\pi$ is an isomorphism over $j_1 (\ol{M})$.
Thus we obtain the following commutative diagram:
\[\xymatrix{
\Gamma_2 \ar[r]^{\pi} &\Gamma_1 \ar[d] \ar[dr]^{p}& \\
\ol{M} \ar[r]^j \ar[ru]^{j_1} \ar[u]^{j_2} & X \ar@{.>}[r]^g & \ol{N}.
}\]

Note that $\ol{f} = p\pi j_2$ by construction. 
Set 
\[
N':=(\ol{N}',(N')^\infty):=(\Gamma_2 , \pi^\ast p^\ast N^\infty ).
\]

Then we have $j_2^\ast (N')^\infty = j_2^\ast \pi^\ast p^\ast N^\infty = \ol{f}^\ast N^\infty = M^\infty$, where the last equality follows from the minimality of $f$.
Therefore, $j_2$ induces a minimal morphism $J:M \to N'$.
Moreover, denoting by $E$ the exceptional divisor of the blow-up $\pi$, we have $E \leq (N')^\infty$ by construction. 
Thus, we obtain a decomposition $(N')^\infty = ((N')^\infty - E) + E$ as a sum of two effective Cartier divisors. 
Therefore, we have $(J:M \to N') \in \Comp (M)$.

By construction $p\pi$ induces a minimal morphism $h:N' \to N$. 
Moreover, it induces an isomorphism $h^\o : (N')^\o \xrightarrow{\sim} N^\o$.
Therefore, $h$ is an isomorphism in $\MSm$.
The commutativity of the above diagram shows $hJ = f$, hence $h^{-1}f=J \in \Comp (M)$.
This finishes the proof.
\end{proof}

\subsection{A key proposition}

\begin{prop}\label{prop:partial-cofinality}
Let $S$ be an $\ulMVfin$-square, and let $(S \to T) \in \Comp (S)$ be any object.
Then there exists a partial compactification $S \to S'$, an object $(S' \to T') \in \Comp (S')$ 
and a morphism $T' \to T$ in $\MSm^\fin$ such that diagram 
\[\xymatrix{
S' \ar[r] & T' \ar[d] \\
S \ar[r] \ar[u] & T
}\]
commutes. 
If $S$ is normal, we can choose $S'$ to be normal.
\end{prop}

\begin{proof}
The first step of the proof is to find a partial compactification $S \to S'$ such that the morphism $S \to T$ extends to a morphism $S' \to T$ in $(\ulMSm^\fin )^\Sq$. 

By Theorem \ref{ex-pc},  take a partial compactification $S \to S'_1$ such that $(S(11) \to S'_1(11))$ dominates $(S(11) \to T(11))$ in $\Comp (S(11))$ and $S'_1(11) \to T(11)$ is minimal. 
For all $(ij) \in \Sq$, let $\Gamma_{ij}$ denote the graph of the rational map $\ol{S}'_1(ij) \dashrightarrow \ol{T}(ij)$. Note that the projection $\Gamma_{ij} \to \ol{S}'_1(ij)$ is a proper birational morphism which is an isomorphism over $\ol{S}(ij) = \ol{S}'_1(ij) \times_{\ol{S}'_1(11)} \ol{S}(11)$.
Consider the \'etale morphism 
\[
f : \bigsqcup_{(ij) \in \Sq} \ol{S}'_1(ij) \to \ol{S}'_1(11).
\]

By Lemma \ref{lem:b-up}, 
we can find a closed subscheme $F$ of $\ol{S}'_1 (11)$ which is supported on $|S'_1 (11)^\infty |$ and such that the base change of $f$ along the blow up $\Bl_{F_1} \ol{S}'_1(11) \to \ol{S}'_1(11)$ factors through $\sqcup_{i,j} \Gamma_{ij}$. 

Set $S'_2:= (S'_1)_F$, where the right hand side is defined as in Lemma \ref{lem:pc-modif}.
Then $S'_2$ is an $\ulMVfin$-square, and the morphism $S \to S'_1$  lifts uniquely to a partial compactification $S \to S'_2$.
Moreover, we have a morphism $\ol{S}'_2 \to \ol{T}$ in $\Sch^\Sq$ by construction.

Since $(S(11) \to S'_2(11)) \in \Comp (S(11))$, there exists an effective Cartier divisor $D_{11}$ on $\ol{S}'_2(11)$ such that $|D_{11}| = \ol{S}'_2(11) - \ol{S}(11)$.
Since $\ol{S}(ij) = \ol{S}'_2(ij) \times_{\ol{S}'_2(11)} \ol{S}(11)$ by definition of partial compactification, if we set $D_{ij} := D_{11} \times_{\ol{S}'_2(11)} \ol{S}'_2(ij)$, we have $|D_{ij}| = \ol{S}'_2(ij) - \ol{S}(ij)$.
By Lemma \ref{lem:increasing}, there exists a positive integer $n$ such that the morphism $\ol{S}'_2(ij) \to \ol{T}(ij)$ induces an admissible morphism 
\[
S'_3(ij) := (\ol{S}'_2(ij), S'_2(ij)^\infty + n D_{ij}) \to T(ij)
\]
for each $(ij) \in \Sq$, and they induce minimal morphisms $S(ij) \to S'_3 (ij)$. 
Since $\ol{S}'_3(ij) = \ol{S}'_2(ij)$ and since $S'_3(ij)^\infty \cap \ol{S}(ij) = S(ij)^\infty$ for each $(ij) \in \Sq$  by construction, the morphism $S \to S'_2$ lifts uniquely to a partial compactification $S \to S'_3$.

We set $S':=S'_3$. 
Then $(S \to S' )\in \Comp (S)$ and the morphism $S \to T$ lifts to a morphism $S' \to T$ in $\ulMSm^\fin$ by construction, as required. 
If $S$ is normal, we replace $S'$ by its normalization as in Corollary \ref{cor:pc-normal}.

Take now any $(S' \to T'_1) \in \Comp (S')$.
Then the graph trick shows that there exists $(S \to T'_2) \in \Comp (S')$ which dominates $S \to T'_1$ and the composite maps 
\[
\ol{h}_{ij} : \ol{T}'_2(ij) \to \ol{T}'_1(ij) \dashrightarrow \ol{T}(ij)
\] are morphisms of schemes. 

For each $(ij) \in \Sq$, let $F_{ij}$ be an effective Cartier divisor on $\ol{T}'(ij)$ such that $|F_{ij}| = \ol{T}'(ij) - \ol{S}'(ij)$, which exists by definition of $\Comp$.
By 
Lemma \ref{lem:increasing}, noting that $\ol{h}_{ij}^\ast T^\infty (ij) \cap \ol{S}'(ij) \leq S'^\infty (ij) = T'^\infty (ij) \cap \ol{S}'(ij)$, there exist positive integers $m_{ij}$ such that we have 
\[
\ol{h}_{ij}^\ast T^\infty (ij) \leq T'^\infty (ij) + m_{ij} F_{ij}
\]
for all $i,j \in \{0,1\}$. 
Define 
\[ T'_{m_{ij}} := (\ol{T}'(ij), T'^\infty (ij) + m_{ij} F_{ij}).\]
Clearly by choosing $\bm{m}=(m_{ij})_{i,j \in \{0,1\}}$ appropriately we obtain admissible morphisms $T'_{m_{ij}} \to T'_{m_{ij}}$ for any morphism $(i,j) \to (i',j')$ in the diagram category $\Sq$.
Therefore, we obtain a diagram
\[T':=T'_{\bm{m}}: \vcenter{\xymatrix{
T'_{m_{00}} \ar[r] \ar[d] & T'_{m_{01}} \ar[d] \\
T'_{m_{10}} \ar[r] & T'_{m_{11}}
}}\]
which belongs to $\Comp(S')$, and its image in $\Comp (S)$ dominates $T$.
This finishes the proof.
\end{proof}

\subsection{End of proof of Theorem \ref{main:cofinality-MV}}
Let $S$ be a normal $\ulMVfin$-square, and $(j:S \to T) \in \Comp (S)$ a compactification. 
Apply Proposition \ref{prop:partial-cofinality}.
By Theorem \ref{main:cofinality-MV-partial}, there exists $(j'_1:S' \to T'_1) \in \Comp^{\MV} (S')$ which dominates $S' \to T'$.
Thus we have the following commutative diagram in $\ulMSm$:
\[\xymatrix{
& T'_1 \ar[d] \\
S' \ar[r]^{j'} \ar[ru]^{j'_1} & T' \ar[d]\\
S\ar[r]^j \ar[u]^f & T
}\]

Note that $j'_1 f$ is ambient and minimal, and it induces an isomorphism $S^\o \xrightarrow{\sim} (T'_1)^\o$ by assumption. 
Therefore, by Lemma \ref{lem:comp-modif}, there exists an isomorphism $g:T'_1 \xrightarrow{\sim} T'_2$ in $\MSm$ such that $(gj'_1f :S \to T'_2) \in \Comp (S)$.
Since $T'_1$ is an $\MV$-square by construction, so is $T'_2$.
Moreover, $S \to T'_2$ dominates $S \to T$ in $\Comp (S)$ by construction. 
This finishes the proof.

\section{Continuity and cocontinuity}\label{section:cocontinuity}

In this section, we prove Theorem \ref{thm:cocontinuity}.

\subsection{Continuity} Note that $\tau_s$, $\ulomega_s$ and $\omega_s$ all preserve fibre products, as required in Proposition \ref{pA.1} (1).
The continuity of $\ulomega_s$ and $\lambda_s$ is obvious by Lemma \ref{lemma:cocontinuity-cd} a), since they send distinguished squares to distinguished squares. This implies the cocontinuity of $\lambda_s$, by Proposition \ref{pA.2} (2).
In the case of $\tau_s$, by the same lemma we must show that $\{\tau_s T(01),\tau_s T(10)\}$ is an $\ulMV$-cover of $\tau_s T(11)$ for any $\MV$-square $T=(T(ij))$. By \cite[Exp. II, Th. 4.4]{SGA4}, this is the case if and only if, for any sheaf of sets $F$ on $\ulMSm$, the map
\[F(\tau_s (T(11)))\to F(\tau_s (T(01)))\times F(\tau_s (T(10)))\]
is injective.
By Condition (2) of Definition \ref{def:new-MV}, there is an $\ulMV$-square $S$ mapping to $\tau_s T$ and such that $S(11)\iso  \tau_s T(11)$, hence the conclusion.
But Condition (2) of Definition \ref{def:new-MV} says that this morphism is dominated by a morphism $S(01)\sqcup S(10)\to \tau_s (T(11))$ for some $\ulMV$-square $S$ such that $S(11)\iso  \tau_s(T(11))$, hence the conclusion. 
Finally, $\omega_s$ is continuous as a composition of continuous functors. 

\subsection{Cocontinuity} 
It suffices to check the conditions of Lemma \ref{lemma:cocontinuity-cd} for $\ulomega_s$ and $\omega_s$.
The Nisnevich cd-structure on $\Sm$ is complete
\cite[Th. 2.2]{unstableJPAA}. 
Therefore, Condition (1) holds for $\ulomega_s$ and $\omega_s$. We now check that the both satisfy Condition (2).

Take any $M \in \ulMSm$ and take any elementary Nisnevich square $S$ in $\Sm$ such that $\ulomega_s (M) \cong S(11)$. 
Take any $M_c \in \Comp (M)$. Then, regarding $S$ as a square in $\ulMSm$, Corollary \ref{cor:pc-normal} implies that there exists a partial compactification $S \to S'$ such that $S'$ is normal, $S'(11) \in \Comp (S(11)) = \Comp (M^\o)$ dominates $M_c \in \Comp (M) \subset \Comp (M^\o)$ and the morphism $S'(11) \to M_c$ is minimal. In particular, $S'(11) \to M_c$ is an isomorphism in $\MSm$.

\subsubsection{Case of $\protect\ulomega_s$}
Define a square $S_1$ in $\ulMSm$ as follows: for any $(ij) \in \Sq$, set
\begin{align*}
\ol{S}_1 (ij) &= \ol{S}'(ij) \times_{\ol{M}_c } \ol{M} ,\\
S_1^\infty (ij) &= S'^\infty (ij) \times_{\ol{M}_c } \ol{M} ,\\
S_1 (ij) &= (\ol{S}_1(ij),S_1^\infty (ij)).
\end{align*}

Then $S_1$ is an $\ulMVfin$-square (hence an $\ulMV$-square).
Moreover, the natural morphism $S_1 (11) \to M$ is an isomorphism in $\ulMSm$ since it is minimal, $\ol{S}_1 (11) \to \ol{M}$ is proper surjective and $S_1^\o (11) = S^\o(11) = M^\o$.
Therefore, the square $S'_1$ obtained by replacing $S_1(11)$ in $S_1$ with $M$ is also an $\ulMV$-square, and we have $\ulomega_s (S'_1)=S$. 
Thus, Condition (2) of Lemma \ref{lemma:cocontinuity-cd} holds for $\ulomega_s$, and therefore $\ulomega_s$ is cocontinuous. 

\subsubsection{Case of $\protect\omega_s$}

Here we take $M\in \MSm$, hence $M=M_c$. 
By Theorem \ref{main:cofinality-MV},
the category $\Comp^{\MV} (S')$ is cofinal in $\Comp (S')$. 
In particular, it is non-empty. Take any object $T \in \Comp^{\MV} (S')$.
Then $T$ is by definition an $\MV$-square such that $T(11) = S'(11) \cong M$.
Therefore, the square $T'$ obtained by replacing $T(11)$ in $T$ with $M$ is also an $\MV$-square, and we have $\omega_s (T') = S$. 
Thus, $\omega_s$ satisfies Condition (2) in Lemma \ref{lemma:cocontinuity-cd}.

This finishes the proof of Theorem \ref{thm:cocontinuity}.

\begin{remark}\label{r5.1} The functor $\tau_s$ is not cocontinuous: take $M\in \MSm$ of dimension $1$. Cover $\tau_s M$ by two affine opens $S(01),S(10)$ (with the minimal modulus structure). Let $f:N\to M$ be a morphism in $\MSm$. If $\tau_s f$ factors through $S(01)$ or $S(10)$, then its image is finite since $\ol{N}$ is proper. But any term of a cover of $N$ is surjective on the ambient spaces: contradiction. \end{remark}

\appendix

\section{Continuous and cocontinuous functors}\label{sec:coconti}

\subsection{Review of the notions}\label{s:co}
 We write $\hat \sC$ for the category of pre\-sheaves of sets on a category $\sC$. If $\sC$ is a site, we write $\tilde{\sC}\subset \hat\sC$ the full subcategory of sheaves of sets, $i_\sC$ for this inclusion and $a_\sC$ for its left adjoint (sheafification).
 
In this subsection, we recall some facts from \cite[Exp. III]{SGA4}, where all sites are assumed to be ``$U$-sites'' in the sense of \cite[II.3.0.2]{SGA4}. We implicitly make this assumption below; it is automatic for sites defined by a cd-structure.

\begin{definition}[\protect{\cite[III.1.1 and III.2.2]{SGA4}}]\label{def:coconti}
Let $\sC$ and $\sD$ be sites, 
and $u : \sC \to \sD$ a functor. We say that $u$ is \emph{continuous} (resp. \emph{cocontinuous}) if the functor $u^*:\hat\sD\to \hat\sC$ (resp. $u_*:\hat\sC\to \hat \sD$) carries sheaves to sheaves.
\end{definition}

\begin{prop}\label{pA.1} \
\begin{enumerate}
\item \cite[III.1.6]{SGA4}.
Suppose that $u$ preserves the fibre products involved in base changes under morphisms coming from covering families of $\sC$. Then $u$ is continuous 
if and only if, for any cover $\{ U_i \to X \}$ in $\sC$,
$\{ u(U_i) \to u(X) \}$ is a cover in $\sD$.
\item The functor $u$ is cocontinuous
if and only if, for any $X \in \sC$
and any cover $\{ V_i \to u(X) \}$ in $\sD$,
there is a cover $\{ U_j \to X \}$ in $\sC$
such that $\{ u(U_j) \to u(X) \}$ refines $\{ V_i \to u(X) \}$ 
(i.e. for each $j$, 
$u(U_j) \to u(X)$ factors through $V_i$ for some $i$).
\end{enumerate}
\end{prop}

\begin{proof}[Proof of (2)] By \cite[III, 2.1 and 2.2]{SGA4}, $u$ is cocontinous if and only if, for any $Y\in \sC$ and any covering sieve $R$ of $u(Y)$, the sieve of $Y$ generated by the arrows $Z\to Y$ such that $u(Z)\to u(Y)$ factors through $R$ is a covering sieve. This condition is clearly equivalent to the one stated in (2).
\end{proof}

\begin{prop}\label{pA.2} \
\begin{enumerate}
\item  \cite[III.1.3]{SGA4} If $u$ is continuous, the functor $u^t:\tilde\sD\to \tilde\sC$ given by Definition \ref{def:coconti} has a left adjoint $u_t$, given by the formula $u_t = a_\sC u_! i_\sD$.
\item  \cite[III.2.5]{SGA4} Let $v$ be left adjoint to $u$. Then $u$ is continuous if and only if $v$ is cocontinuous.
\end{enumerate}
\end{prop}

\subsection{The case of cd-structures}

\begin{lemma}\label{lemma:cocontinuity-cd}
Let $\mathcal{C}$ and $\mathcal{D}$ be sites and let $u : \mathcal{C} \to \mathcal{D}$ be a functor. 
Assume that the topologies on $\mathcal{C}$ and $\mathcal{D}$ are generated by \textit{cd}-structures $P_{\mathcal{C}}$ and $P_{\mathcal{D}}$, respectively. \\
a) Assume that $P_\sC$ is complete and that $u$ verifies the condition of Proposition \ref{pA.1}. Then $u$ is continuous if and only if it sends elementary covers to covers.  \\
b) Assume that
\begin{enumerate}
\item $P_{\mathcal{D}}$ is complete.
\item For any $X' \in \mathcal{C}$ and for any distinguished square $Q \in P_{\mathcal{D}}$ such that $Q(11)=u(X')$, there exists $Q' \in P_{\mathcal{C}}$ such that $u(Q') = Q$ and $Q'(11) = X'$.
\end{enumerate}
Then $u$ is cocontinuous.
\end{lemma}

\begin{proof} a) Necessity is obvious. 
Sufficiency: since $P_\sC$ is complete, any cover can be refined by a simple cover as in \cite[Def. 2.2]{cdstructures} (this is the definition of complete, see loc. cit., Def. 2.3). Therefore, it suffices to show that $u$ sends simple covers  to covers. For any simple cover $\sV=\{V_l \to X\}_{l \in L}$ of an object $X\in \sC$, define $n_\sV := |L|$ (note that $n_\sV$ is finite).  We prove the assertion by induction on $n_\sV \geq 1$. 

If $n_\sV = 1$, then $\sV$ consists of an isomorphism since if a simple cover is obtained by composing an elementary cover at least once, then the cardinality of the indexing set must be $>1$. Therefore the statement is trivial.

Assume that $n_\sV > 1$ and take a distinguished square $Q \in P_\sC$ of the form 
\[\xymatrix{
B \ar[r] \ar[d]&  Y \ar[d]^p \\
A \ar[r]^e & X
}\]
and simple covers $\mathcal{Y} = \{p_s : Y_s \to Y\}_{s \in S}$ and $\mathcal{A} = \{q_t : A_t \to A\}_{t \in T}$ such that $\mathcal{V}=\{p \circ p_s , e \circ q_t \}_{s \in S, t \in T}$.
Then obviously we have $n_{\mathcal{Y}}< n_{\mathcal{V}}$ and $n_{\mathcal{A}} < n_{\mathcal{V}}$, and we are done by induction.

b) Take any $X' \in \mathcal{C}$. 
By the same argument as in a), it suffices to show the following assertion: for any simple cover $\mathcal{V}=\{V_j \to X\}_{j\in J}$ there exists a cover $\mathcal{U}=\{U_i \to X'\}_{i \in I}$ such that $u(\mathcal{U})$ refines $\mathcal{V}$. We proceed by induction on $n_{\mathcal{V}} \geq 1$, the case $n_{\mathcal{V}} = 1$ being trivial as above.

Assume that $n_{\mathcal{V}} > 1$. With the same notation as in a), but with $Q \in P_{\mathcal{D}}$, 
by (2), there exists $Q'\in P_{\mathcal{C}}$ of the form 
\[\xymatrix{
B' \ar[r] \ar[d]&  Y' \ar[d]^{p'} \\
A' \ar[r]^{e'} & X' 
}\]
such that $u(Q')=Q$. In particular, we have $u(Y')=Y$ and $u(A')=A$. 
By the induction hypothesis, there exist covers $\mathcal{Y}'=\{p'_{s'} : Y'_{s'} \to Y'\}_{s' \in S'}$ and $\mathcal{A}' = \{q'_{t'} : A'_{t'} \to A'\}_{t' \in T'}$ such that $u(\mathcal{Y}')$ and $u(\mathcal{A}')$ refine $\mathcal{Y}$ and $\mathcal{A}$, respectively. 
Then we obtain a cover $\mathcal{U}=\{p' \circ p'_{s'} , e' \circ q'_{t'}\}_{s'\in S',t' \in T'}$ of $X'$ such that  $u(\mathcal{U})=\{p \circ u(p'_{s'}) , e \circ u(q'_{t'})\}_{s'\in S',t' \in T'}$ refines $\mathcal{V}$ by construction. This finishes the proof.
\end{proof}

\section{A criterion for an open immersion} 

In this appendix, we state an important result that we found in \cite[Tag:081M]{stack}\footnote{\url{https://stacks.math.columbia.edu/tag/081M}} (see also \cite[Cor. 2.2]{lazard} and \cite[Th. 2.7]{oda} in the affine case). We include its proof for completeness.

\setcounter{subsection}{1}

\begin{thm}\label{strong-lemma} 
Let $f : X \to S$ be a morphism of schemes with $X$ Noetherian, and let $U \subset S$ an open subset.
Assume the following conditions hold:
\begin{thlist}
\item  $f$ is separated, locally of finite presentation, and flat.
\item $f^{-1} (U) := U \times_S X \to U$ is an isomorphism.
\item The inclusion $U \to S$ is quasi-compact and scheme-theoretically dense.
\end{thlist} 
Then, $f$ is an open immersion. 
\end{thm}

\begin{proof}
First, consider the case that $f$ is finite. Then, the proof is easy. Indeed, since $f$ is finite flat of finite presentation, it is finite locally free.
Since $f$ is an isomorphism over a dense open subset, its degree is equal to $1$, hence it is an isomorphism everywhere.

\ 

Next, we treat the general case.
Since $f$ is flat of finite type, the image $f(X) \subset S$ is open. Therefore, we may assume that $f$ is surjective.
We want to prove that $f$ is an isomorphism. 
By the above argument, it suffices to prove that $f$ is finite.

\begin{claim}
$f$ is quasi-finite.
\end{claim}

\begin{proof}
Since the problem is local on $S$, we may assume that $S$ is affine.
Then, since $X$ is quasi-compact by assumption, the morphism $f$ is quasi-compact. 
Moreover, $f$ is of relative dimension $0$ since $f$ is flat and birational.
Therefore, $f$ is quasi-compact and locally quasi-finite, hence quasi-finite.
\end{proof}

We need the following propositions from \cite[\S 2.3 Prop. 8 (a), \S 2.5 Prop. 2]{neronmod}:
\begin{prop}[\'etale localization of quasi-finite morphisms]\label{et-loc-of-qfin}
Let $f : X \to Y$ be locally of finite type. Let $x$ be a point of $X$, and set $y:=f(x)$.

If $f$ is quasi-finite at $x$, then there exists an \'etale neighborhood $Y' \to Y$ of $y$ such that the morphism $f' : X' \to Y'$, obtained from $f$ by the base change $Y' \to Y$ induces a finite morphism $f' |_{U'} : U' \to Y'$, where $U'$ is an open neighborhood of the fiber of $X' \to X$ above $x$.
In addition, if $f$ is separated, $U'$ is a connected component of $X'$. \qed
\end{prop}

\begin{prop}[compatibility between schematic images and flat base changes]\label{sch-im-flat-bc}
Let $f : X \to Y$ be an $S$-morphism which is quasi-compact and quasi-separated.
Let $g : S' \to S$ be a flat morphism, and denote by $f':X'\to Y'$ be the $S'$-morphism obtained from $f$ by base change.
Let $Z$ (resp. $Z'$) be the schematic image of $f$ (resp. $f'$).
Then, $Z \times_S S'$ is canonically isomorphic to $Z'$. \qed
\end{prop}

Since the finiteness of $f$ is Zariski local on $S$, it suffices to check it over an open neighborhood of a fixed point $s \in S$.
Take a point $x \in X$ above $s$.
Take an \'etale neighborhood $g : S' \to S$ of $s$ as in Proposition \ref{et-loc-of-qfin}, and set $X' := X \times_S S'$.
Denote by $f'$ the induced morphism $X' \to S'$.
Since $f$ is separated, quasi-finite and locally of finite type, there exists a connected component $V' \subset X'$ such that 
\[
f'|_{V'} : V' \to X' \to S'
\]
is finite, and $V'$ is an open neighborhood of the fiber of $x$. 
Since $f$ is flat, so is $f'$, hence the image $f'(V') \subset S'$ is an open subset.
By shrinking $S'$, we may assume that $V' \to S'$ is surjective.
Since $f$ is an isomorphism over $U \subset S$, $f'$ is an isomorphism over $g^{-1} (U) \subset S'$.
Therefore, combining with the surjectivity of $f'|_{V'}$, we have 
\begin{equation}\label{eq-a}
(f')^{-1} (g^{-1} (U)) \subset V' .
\end{equation}
On the other hand, since the map $g \circ f'$ is a flat morphism, Proposition \ref{sch-im-flat-bc} implies that the open subset
\begin{equation}\label{eq-b}
(f')^{-1} (g^{-1} (U)) \subset X' = V' \sqcup X'_1 ,
\end{equation}
is schematically dense, where $X'_1$ is an open and closed subset of $X'$.
Therefore,  (\ref{eq-a}) shows that $X'_1 = \emptyset$ and $X' = V'$.
Therefore, we have $f' = f' |_{V'}$, hence $f'$ is finite.

By replacing $S$ by the image of $S' \to S$, we may assume that $S' \to S$ is an fpqc-cover.
Since finiteness is an fpqc-local property, we conclude that $f$ is finite. This finishes the proof of Theorem \ref{strong-lemma}.
\end{proof}

\enlargethispage*{20pt}

\section{Cofilteredness for diagrams}
\setcounter{subsection}{1}

Let $u:\sC\to \sD$ have a pro-left adjoint $v:\sD\to \pro{}\sC$. We give ourselves a system of subcategories $(I(d)\subset d\downarrow u)_{d\in \sD}$  representing $v$. Thus each $I(d)$ is ordered and cofiltered, and $v(d)= ``\lim"_{c\in I(d)} c$ for any $d$.

\begin{lemma}\label{l2} Let $\Delta$ be a finite category without loops: the collections of objects and morphisms of $\Delta$ are finite and the only endomorphisms of objects are the identities.
For $d\in \sD$, define a subcategory $I(\ul{d})$ of $\ul{d}\downarrow u^\Delta$ as follows: an object $X$ (resp. morphism $f$) of $\ul{d}\downarrow u^\Delta$ is in $I(\ul{d})$ if and only if $X(\delta)$ (resp. $f(\delta)$) is in $I(d(\delta))$ for all $\delta\in \Delta$. Then the category $I(\ul{d})$ is ordered and cofiltered for all $\ul{d}\in \sD^\Delta$.
\end{lemma}

\begin{proof} Ordered is obvious. For cofiltered, induction on $\# Ob(\Delta)$. We may assume $\Delta$ nonempty. The finiteness and ``no loop'' hypotheses imply that $\Delta$ has an object $\delta_0$ such that no arrow leads to $\delta_0$: we call such on object \emph{minimal}. Let $\Delta'$ be the subcategory of $\Delta$ obtained by removing $\delta_0$ and all the arrows leaving from $\delta_0$.  Let $X_1:\ul{d}\to u^\Delta(\ul{c_1})$, $X_2:\ul{d}\to u^\Delta(\ul{c_2})$ be two objects of $I(\ul{d})$. By induction, we may find $Y_3:\ul{d}\mid \Delta'\to u^{\Delta'}(\ul{c_3}')\in I(\ul{d}\mid \Delta')$ sitting above $X_1\mid \Delta'$ and $X_2\mid \Delta'$.  Let $f:\delta_0\to \delta$ be an arrow, with $\delta\in \Delta'$: by the functoriality of $v$, there exists a commutative diagram in $\sD$
\[\begin{CD}
d(\delta_0)@>\phi(f)>> u(c(f))\\
@V{d(f)}VV @Vu(\psi(f))VV\\
d(\delta)@>Y_3(\delta)>> u(c'_3(\delta))
\end{CD}\]
with $\phi(f)\in I(d(\delta_0))$. Since $I(d(\delta))$ is cofiltered, we may find an object $d(\delta_0)\by{g}u(c)\in I(d(\delta))$ sitting above all $\phi(f)$'s as well as  $X_1(\delta_0)$ and $X_2(\delta_0)$. Then, together with $Y_3$,  $X_3(\delta_0)=:g$ completes the construction of $X_3$ dominating $X_1$ and $X_2$.
\end{proof}


\end{document}